%% file: tvq.tex
\documentclass[a4paper,11pt]{plan-base}
\usepackage{tikz}
\usepackage{stmaryrd}
\usepackage{algorithmic}
\usepackage{float}
\usepackage[misc,geometry]{ifsym}

\input{texbase}

\input{symbols-common}

\input{symbols-ten}
\input{todo}

\input{tvq-syms}

\title{Limiting aspects of non-convex $\TVphi$ models}

\author{%
    Michael Hinterm\"uller\thanks{
        Institute for Mathematics,
        Humboldt University of Berlin, 
        Germany.
        \\
        E-mail: \texttt{hint@math.hu-berlin.de}
    },
    \ %
    Tuomo Valkonen\thanks{
        Department of Applied Mathematics and Theoretical Physics,
        University of Cambridge, United Kingdom.
        \\
        E-mail: \texttt{tuomo.valkonen@iki.fi}.
    },\ \!\!${ }^\text{, \Letter}$,
    \ %
    and
    \ %
    Tao Wu\thanks{
        Institute for Mathematics and Scientific Computing,
        University of Graz, 
        Austria.
        \\
        E-mail: \texttt{tao.wu@uni-graz.at}
    }
}

\begin{document}

\maketitle

\begin{abstract}
    Recently, non-convex regularisation models have been introduced
    in order to provide a better prior for gradient distributions 
    in real images. They are based on using
    concave energies $\phi$ in the total variation type functional
    $\TV^\phi(u) \defeq \int \phi(|\nabla u(x)|) \d x$.
    In this paper, it is demonstrated that for typical choices of $\phi$, functionals
    of this type pose several difficulties when extended to the entire
    space of functions of bounded variation, $\BVspace(\Omega)$. 
    In particular, if $\phi(t)=t^q$ for $q \in (0, 1)$ and $\TV^\phi$ 
    is defined directly for piecewise constant functions
    and extended via weak* lower semicontinuous envelopes
    to $\BVspace(\Omega)$, then still $\TV^\phi(u)=\infty$
    for $u$ not piecewise constant. If, on the other hand,
    $\TV^\phi$ is defined analogously via continuously 
    differentiable functions, then $\TV^\phi \equiv 0$, (!).
    We study a way to remedy the models through
    additional multiscale regularisation and area strict
    convergence, provided that the energy $\phi(t)=t^q$ is 
    linearised for high values. The fact, that this kind of energies
    actually better matches reality and improves reconstructions,
    is demonstrated by statistics and numerical experiments.
    

    
\end{abstract}

\section{Introduction}

Recently introduced non-convex total variation models are based on employing 
concave energies $\phi$, in discrete versions of functionals of the form
\begin{equation}
    \label{eq:tvqc}
    \TVphic(u) \defeq \int_\Omega \phi(|\grad u(x)|) \d x,
    \quad (u \in C^1(\Omega)),
\end{equation}
which we call the \term{continuous model}, or
\begin{equation}
    \label{eq:tvqd}
    \TVphid(u) \defeq \int_{J_u} \phi(\abs{u^+(x)-u^-(x)}) \d \H^{m-1}(x),
    \quad (u \text{ piecewise constant}),
\end{equation}
which we call the \term{discrete model}. Here $\Omega \subset \R^m$ is our image domain, 
and $J_u$ is the jump set of $u$, where the one-sided traces $u^\pm$ from different
sides of $J_u$ differ.
The typical energies include, in particular, $\phi(t)=t^q$ for $q \in (0, 1)$.
The models based on discretisations of \eqref{eq:tvqd} have been proposed for
the promotion of piecewise constant (cartoon-like) images
\cite{geman1984stochastic,nikolova2002minimizers,nikolova2008efficient,chen2010smoothing},
whereas models based on discretisations of \eqref{eq:tvqc} have been proposed for
the better modelling of gradient distributions in real-life images
\cite{huang1999statistics,HiWu13_siims,HiWu14_coap,ochsiterated}. 
To denoise an image $z$, one may then solve the nonconvex Rudin-Osher-Fatemi type problem
\begin{equation}
    \label{eq:tvq-rof}
    \min_u \frac{1}{2} \norm{z-u}^2 + \alpha \TVphi(u)
\end{equation}
for $\TVphi=\TVphic$ or $\TVphi=\TVphid$.
Observe that \eqref{eq:tvqc} is only defined rigorously for differentiable 
functions. In contrast to \eqref{eq:tvqd}, it is in particular not defined for
piecewise constant discretisations, or images with discontinuities. 
The functional has to be extended to the whole space of functions of bounded
variation denoted by $\BVspace(\Omega)$, see \cite{giusti1984minimal} for its definition, in order to obtain a sound model in the
non-discretised setting. Alternatively, we may take \eqref{eq:tvqd}, defined 
for piecewise constant functions, as the basis and extend it to continuous 
functions. We will study the extension of both models
\eqref{eq:tvqc} and \eqref{eq:tvqd} to $\BVspace(\Omega)$. We demonstrate 
that \eqref{eq:tvqc} in particular has severe theoretical difficulties 
for typical choices  of $\phi$. We also demonstrate that some of these
difficulties can be overcome by altering the model to better match reality,
although we also need additional multiscale regularisation in the model for theoretical purposes.

Let us consider the discrete model $\TVphid$ first.
We assume that we have a regularly spaced grid $\Omega_h \subset \Omega \isect h\Z^m$,
($h > 0$), and a function $u_h: \Omega_h \to \R$. By $\{e_i\}_{i=1}^m$ we denote the canonical
orthonormal basis of $\R^m$. Then we identify $u_h$ with a function $u$ that
is constant on each cell $k + [0, h]^m$, ($k \in \Omega_h$). Accordingly, we have
\begin{equation}
    \label{eq:tvq-disc}
    \TVphid(u_h) \defeq \sum_{k \in \Omega_h} \sum_{i=1}^m h^{m-1} \phi(|u_h(k+e_i)-u_h(k)|).
\end{equation}
This discrete expression with $h=1$ is essentially what is studied 
in \cite{nikolova2002minimizers,nikolova2008efficient,chen2010smoothing},
although \cite{nikolova2002minimizers} studies also more general discrete models.
In the function space setting, this model has to be extended to all of $\BVspace(\Omega)$,
in particular to smooth functions. The extension naturally has to be lower 
semicontinuous in a suitable topology, in order to guarantee the existence of 
solutions to \eqref{eq:tvq-rof}. 
Therefore, one is naturally confronted with the question whether such an 
extension can be performed meaningfully?

Let us consider a simple motivating example on $\Omega=(0, 1)$ with $\phi(t)=t^q$ 
for $q \in (0, 1)$. We aim to approximate the ramp function
\[
    u(t)=t
\]
by piecewise constant functions. Given $k > 0$, we thus define
\[
    u^k(t)=
    i/k, \quad \text{for~}t \in [(i-1)/k, i/k) \text{~and~}i\in\{1,\ldots,k\}.
\]
Clearly 
we have that $u^k$ converges strongly to $u$ in $L^1(\Omega)$.
Using the discrete model \eqref{eq:tvq-disc} with $h=1/k$, one has
\[
    \TVqd(u^k)=\sum_{i=1}^k (1/k)^0 \cdot (1/k)^q = k^{1-q}.
\]
We see that $\lim_{k \to \infty} \TVqd(u^k) = \infty$! This suggests that the $\TVq$ 
model based on the 
discrete functional might only allow piecewise constant functionals. In other words,
$\TVqd$ would induce pronounced staircasing -- a property desirable when restoring
piecewise constant images, but less suitable for other applications. 
In Section \ref{sec:tvphid}, we will indeed demonstrate that either $u$ is piecewise constant, 
or $u \not\in \BVspace(\Omega)$.

In order to highlight the inherent difficulties, let us then consider the continuous model $\TVphic$, directly given by \eqref{eq:tvqc} 
for differentiable functions. In particular, \eqref{eq:tvqc} also serves as a definition 
of $\TVphic$ for continuous piecewise affine discretisations of $u \in C^1(\Omega)$.
We observe that  if $u \in C^1(\closure \Omega)$ on a bounded
domain $\Omega$, and we set $u_h(k)=u(k)$ for $k \in \Omega_h$, then
\begin{equation}
    \label{eq:tvq-cont-disc}
    \TVphi_{\mathrm{c},h}(u_h) \defeq \sum_{k \in \Omega_h} h^m \phi(|\grad_h u_h(k)|),
    \quad
    \text{with}
    \quad
    [\grad_h u_h(k)]_i \defeq \bigl(u_h(k+e_i)-u_h(k)\bigr)/h
\end{equation}
satisfies
\[
    \lim_{h \downto 0} \TVphi_{\mathrm{c},h}(u_h) = \TVphic(u).
\]
This approximate model $\TVphi_h$ with $h=1$ is essentially what is considered
in \cite{HiWu14_coap,HiWu13_siims,ochsiterated}. On an abstract level, it is also 
covered by \cite{nikolova2002minimizers}. The question now is whether the definition 
of $\TVphic$ can be extended to functions of bounded variation in a meaningful manner.

To start our investigation, let us try to approximate on $\Omega=(-1, 1)$ 
the step function
\[
    u(t)=\begin{cases}
            0, & t<0, \\
            1, & t \ge 0.
         \end{cases}
\]
Given $k>0$, we define
\[
    u^k(t)=
    \begin{cases}
            0, & t < -1/k, \\
            1, & t \ge 1/k, \\
            \frac{1}{2}(kt+1), & t \in [-1/k, 1/k).
    \end{cases}
\]
Then $u^k \to u$ in $L^1(\Omega)$. However, the continuous model \eqref{eq:tvqc} 
with $\phi(t)=t^q$ for $q \in (0, 1)$ gives
\[
    \TVqc(u^k) = (2/k)^{q-1}.
\]
Thus $\TVqc(u^k) \to 0$ as $k \upto \infty$. This suggests that any extension 
of $\TVqc$ to $u \in \BVspace(\Omega)$ through weak* lower semicontinuous 
envelopes will have $\TVqc(u)=0$, and that jumps in
general will be free. In Section \ref{sec:tvphic} we will prove
this and something more striking. A weak* lower semicontinuous extension
will necessary satisfy $\TVqc \equiv 0$. 

Despite this discouraging property, after discussing the implications
of the above-mentioned results in Section \ref{sec:discuss},
we find appropriate remedies. Our associated principal approach is given in
Section \ref{sec:areastrict}. It utilizes the (stronger) notion
of \term{area-strict convergence} \cite{delladio1991lower,kristensen2010relaxation},
which -- as will be shown -- can be obtained using the multiscale analysis
functional $\eta$ from \cite{tuomov-bd,tuomov-ap1}. In
Section \ref{sec:alternative} we also discuss alternative remedies which are related to
compact operators and the space $\SBVspace(\Omega)$
of special functions of bounded variation. 
In order to keep the flow of the paper, the pertinent proofs are relegated to the Appendix. 

To show existence of solutions to the fixed $\TVphic$ model involving
area-strict convergence, we require that $\phi$ is level coercive,
i.e.~$\lim_{t \to \infty} \phi(t)/t  > 0$. This induces a linear
penalty to edges in the image. 
Based on these considerations, one arrives at the question 
whether gradient statistics, such as the ones in \cite{huang1999statistics},
are reliable in dictating the prior term (regularizer).
Our experiments
on natural images in Section \ref{sec:stat} suggest that this is not the case. In fact, 
the jump part of the image appears to have different statistics from the smooth part. 
It seems that the
conventional $\TV$ regularization \cite{Rud1992} provides a model for the jump part, which is superior to the nonconvex TV-model.
This statistically validates our model, which is also suitable for a function space setting.
Our rather theoretical
starting point of making the $\TVphic$ model sound in function
space therefore leads to improved practical models.
Finally, in Section \ref{sec:numerical} we
study image denoising with this model, and finish with conclusions
in Section \ref{sec:concl}. We however begin with notation
and other preliminary matters in the following Section \ref{sec:prelim}.



\section{Notation and preliminaries}
\label{sec:prelim}

We write the boundary of a set $A$ as $\BD A$, and the 
closure as $\closure A$. The open ball of radius
$\rho$ centred at $x \in \R^m$ is denoted by $\B(x, \rho)$. 

We denote the set of non-negative reals as $\nonnegR \defeq [0, \infty)$.
If $\phi: \nonnegR \to \nonnegR$, then we write
\[
    \phi_0 \defeq \lim_{t \downto 0} \phi(t)/t,
    \quad
    \text{and}
    \quad
    \phi^\infty \defeq \lim_{t \upto \infty} \phi(t)/t,
\]
implicitly assuming that the (possibly infinite) limits exist.

For $\Omega \subset \R^m$, we denote the space of (signed)
Radon measures on $\Omega$ by $\Meas(\Omega)$, and the space
of $\R^m$-valued Radon measures by $\Meas(\Omega; \R^m)$. 
We use the notation $\abs{\mu}$ for the \term{total variation measure}
of $\mu \in \Meas(\Omega; \R^m)$, and define the \term{total variation 
(Radon) norm} of $\mu$ by
\[
    \norm{\mu}_{\Meas(\Omega; \R^m)} \defeq \abs{\mu}(\Omega).
\]
For a measurable set $A$, we denote by $\mu \restrict A$ the restricted
measure defined by $(\mu \restrict A)(B) \defeq \mu(A \isect B)$.
The restriction of a function $u$ to $A$ is denoted by $u|A$.
On any given ambient space $\R^m$, ($k \le m$), we write $\H^k$ for the
$k$-dimensional \term{Hausdorff measure}, and $\L^m$ for 
the \term{Lebesgue measure}.

If $J \subset \R^m$ 
and there exist Lipschitz maps $\gamma_i: \R^{m-1} \to \R$ with
\[
    \H^{m-1}\left(J \setminus \Union_{i=1}^\infty \gamma_i(\R^{m-1})\right)=0,
\]
then we say that $J$ is \term{countably $\H^{m-1}$-rectifiable}.

We say that a function $u: \Omega \to \R$ on an open domain $\Omega \subset \R^m$ 
is of \term{bounded variation} (see, e.g., \cite{ambrosio2000fbv}
for a thorough introduction), denoted $u \in \BVspace(\Omega)$, 
if $u \in L^1(\Omega)$, and the distributional gradient $D u$, given by
\[
    Du(\phi) \defeq \int_\Omega \divergence \phi(x) u(x) \d x,
    \quad
    (\phi \in C_c^\infty(\Omega)),
\]
is a Radon measure, i.e.~$|Du|$ is finite. In this case, we can decompose $Du$ into
\[
    Du = \grad u \L^n + D^j u + D^c u,
\]
where $\grad u \L^n$ is called the \term{absolutely continuous part},
$D^j u$ the \term{jump part}, and $D^c u$ the \term{Cantor part}.
We also denote the \term{singular part} by
\[
    D^s u \defeq D^j u + D^c u 
\]
The density $\grad u \in L^1(\Omega; \R^m)$ corresponds to the
classical gradient if $u$ is differentiable. The jump part may be
written as
\begin{equation}
    \notag
    D^j u = (u^+ - u^-) \otimes \nu_{J_u} \H^{m-1} \restrict J_u,
\end{equation}
where the \term{jump set} $J_u$ is countably $\H^{m-1}$-rectifiable,
$\nu_{J_u}(x)$ is its normal, and $u^+$ and $u^-$ are one-sided
traces of $u$ on $J_u$. The remaining \term{Cantor part} $D^c u$ 
vanishes on any Borel set which is $\sigma$-finite with respect to $\H^{m-1}$;
in particular $\abs{D^c u}(J_u)=0$.
We declare $u$ an element of the space $\SBVspace(\Omega)$ of
\term{special functions of bounded variation}, if $u \in \BVspace(\Omega)$
and $D^c u = 0$. 

We define the norm
\[
    \norm{u}_{\BVspace(\Omega)} \defeq \norm{u}_{L^1(\Omega)}
        + \norm{D u}_{\Meas(\Omega; \R^m)},
\]
and also denote
\[
    \TV(u) \defeq \norm{D u}_{\Meas(\Omega; \R^m)}.
\]
We say that a sequence $\{u^i\}_{i=1}^\infty \subset \BVspace(\Omega)$, 
\term{converges weakly*} to $u$ in $\BVspace(\Omega)$, denoted by
$u^i \weaktostar u$, if $u^i \to u$ strongly in $L^1(\Omega)$ and
$D u^i \weaktostar D u$  weakly* in $\Meas(\Omega; \R^m)$. 
If in addition $\abs{D u^i}(\Omega) \to \abs{D u}(\Omega)$, 
we say that the convergence is \term{strict}.

\section{Limiting aspects of the discrete $\TVphi$ model}
\label{sec:tvphid}

We begin by rigorously defining and analysing
the discrete $\TVphi$ model \eqref{eq:tvqd} in $\BVspace(\Omega)$. This model is used in
the literature to promote piecewise constant solutions to 
image reconstruction problems. For our analysis we consider
the following class of energies $\phi$.

\begin{definition}
    Define $\Wd$ as the set of increasing, lower semicontinuous,
    subadditive functions  $\phi: \nonnegR \to \nonnegR$ 
    that satisfy $\phi(0)=0$ and $\phi_0=\infty$.
\end{definition}


\begin{example}
    Examples of $\phi \in \Wd$ include $\phi(t)=t^q$ for $q \in [0, 1)$.
\end{example}



\begin{definition}
    Denote by $\pwc(\Omega)$ the set of functions
    $u \in \BVspace(\Omega)$ that are piecewise constant 
    in the sense $D u = D^j u$. We then write
    $\abs{D^j u} = \theta_u \H^{m-1} \restrict J_u$.
\end{definition}

\begin{definition}
Given an energy $\phi \in \Wd$, the ``discrete''
non-convex total variation model is defined by
\[
    \widetilde{\TVphid}(u) \defeq \int_{J_u} \phi(\theta_u(x)) \d \H^{m-1}(x),
    \quad
    (u \in \pwc(\Omega)),
\]
and extend this to $u \in \BVspace(\Omega)$ by defining
\[
    \TVphid(u) \defeq \liminf_{\substack{u^i \weaktostar u, \\ u^i \in \pwc(\Omega)}} \widetilde{\TVphid}(u^i),
\]
with the convergence weakly* in $\BVspace(\Omega)$, in order to obtain
a weak* lower semicontinuous functional.
\end{definition}

The functional $\widetilde{\TVphid}$ in particular agrees with \eqref{eq:tvq-disc}.
Our main result regarding this model is the following.

\begin{theorem}
    \label{thm:tvphid}
    Let $\phi \in \Wd$. 
    Then
    \[
        \TVphid(u)=\infty \quad \text{ for } \quad u \in \BVspace(\Omega) \setminus \pwc(\Omega).
    \]
\end{theorem}

The proof is based on the SBV compactness theorem \cite{ambrosio1989compact}; 
alternatively it can be proved via rectifiability results in the theory 
of currents \cite{white1999rfc}, as used in the study of transportation
networks, e.g., in \cite{paoste05,tuomov-tns}.

\begin{theorem}[SBV compactness \cite{ambrosio1989compact}]
    \label{thm:sbv-compactness}
    Let $\Omega \subset \R^m$ be open and bounded.
    Suppose $\phi, \psi: \nonnegR \to \nonnegR$ are lower semicontinuous and
    increasing with $\phi^\infty=\infty$ and $\psi_0=\infty$.
    Suppose $\{u^i\}_{i=1}^\infty \subset \SBVspace(\Omega)$ 
    and $u^i \weaktostar u \in \SBVspace(\Omega)$ weakly* in $\BVspace(\Omega)$.
    If
    \[
        \sup_{i=1,2,3,\ldots} \left(\int_\Omega \phi(|\grad u^i(x)|) \d x
            + \int_{J_{u^i}} \psi(\theta_{u^i}(x)) \d \H^{m-1}(x)\right) < \infty,
    \]
    then there exists a subsequence of $\{u^i\}_{i=1}^\infty$,
    unrelabelled, such that
    \begin{gather}
        \label{eq:sbv-conv-1}
        u^i \to u \text{ strongly in } L^1(\Omega),
        \\
        \label{eq:sbv-conv-2}
        \grad u^i \weakto \grad u \text{ weakly in } L^1(\Omega;\R^m),
        \\
        \label{eq:sbv-conv-3}
        D^j u^i \weaktostar D^j u \text{ weakly* in } \Meas(\Omega; \R^m).
    \end{gather}
    If, moreover, $\psi$ is subadditive with $\psi(0)=0$, then
    \begin{equation}
        \label{eq:sbv-conv-4}
        \int_{J_{u}} \psi(\theta_{u}(x)) \d \H^{m-1}(x)
        \le
        \liminf_{i \to \infty}
        \int_{J_{u^i}} \psi(\theta_{u^i}(x)) \d \H^{m-1}(x).
    \end{equation}
\end{theorem}

\begin{remark}
    As is typically stated in the SBV compactness theorem, convexity
    of $\psi$ is required for \eqref{eq:sbv-conv-4}.
    The fact, that subadditivity and $\psi(0)=0$ suffices,
    follows from \cite[Chapter 5]{ambrosio1989compact},
    or from mapping to currents and using \cite{white1999rfc}.
\end{remark}

\begin{proof}[Proof of Theorem \ref{thm:tvphid}]
    Given $u \in \BVspace(\Omega)$, let $u^i \in \pwc(\Omega)$
    satisfy $u^i \weaktostar u$ weakly* in $\BVspace(\Omega)$.
    Then the SBV compactness theorem shows that
    $\grad u=\grad u^i=0$ and $D^c u=0$.
    Thus $u \in \pwc(\Omega)$.
\end{proof}

\begin{remark}
    The functions $\phi(t) = \alpha t/(1+\alpha t)$ and 
    $\phi(t) = \log(\alpha t + 1)$ for $\alpha > 0$,
    considered in \cite{nikolova2008efficient} for
    reconstruction of piecewise constant images, 
    do not have the property $\phi(t)/t \to \infty$
    as $t \downto 0$. The above result therefore does
    not apply, and indeed $\TVphid$ defined using these
    functions will not force $u$ with $\TVphid(u) < \infty$ 
    to be piecewise constant, as the following 
    result states.
\end{remark}

\begin{proposition}
    \label{prop:tvphid-remedy}
    Let $\phi: \nonnegR \to \nonnegR$ be continuously differentiable and satisfy $\phi(0)=0$.
    Then the following hold.
    \begin{enumroman}
    \item
    \label{item:tvphid-remedy:1}
    If $\phi_0 < \infty$ and $\phi$ is subadditive, then there exist a constant $C>0$ such that
    \[
        \TVphid(u) \le C\,\TV(u),
        \quad (u \in \BVspace(\Omega)).
    \]
    \item
    \label{item:tvphid-remedy:2}
    If $\phi_0 > 0$ and $\phi$ is increasing, then for every $M>0$
    there exists also a constant $c=c(M) > 0$ such
    that
    \[
        c\,\TV(u) \le \TVphid(u),
        \quad (u \in \BVspace(\Omega),\ \norm{u}_{L^\infty(\Omega)} \le M).
    \]
    \end{enumroman}
\end{proposition}
\begin{proof}
    We first prove the upper bound.
    To begin with, we observe that
    \begin{equation}
        \label{eq:tvphid-remedy-phi1}
        \phi(t) \le \varphi_0 t.
    \end{equation}
    Indeed, since $\phi$ is sub-additive we have
    \[
        \lim_{\delta \downto 0} \frac{\phi(t+\delta)-\phi(t)}{\delta} \le \lim_{\delta \downto 0} \frac{\phi(\delta)}{\delta} = \varphi_0 < \infty
    \]
    Thus $\phi'(t) \le \varphi_0 $. As $\phi(0)=0$, it follows that $\phi(t) \le \varphi_0 t$.
    
    Now, with $u \in \BVspace(\Omega)$, we pick a sequence
    $\{u^k\}_{k=1}^\infty$ in $\pwc(\Omega)$
    converging to $u$ strictly in $\BVspace(\Omega)$;
    for details see \cite{ckp1999regularization}.
    Then by \eqref{eq:tvphid-remedy-phi1} we have
    \[
        \widetilde\TVphid(u^k) \le \varphi_0 \TV(u^k), \quad (k=1,\ldots,\infty).
    \]
    Then, by the definition of $\TVphid(u)$ and the strict convergence
    \[
        \TVphid(u)
        \le \liminf_{k \to \infty} \widetilde\TVphid(u^k)
        \le \liminf_{k \to \infty} \varphi_0 \TV(u^k)
        =\varphi_0 \TV(u).
    \]
    The claim in \ref{item:tvphid-remedy:1} follows.
    
    Let us now prove the lower bound in \ref{item:tvphid-remedy:2}.
    First of all, we observe the existence of $c>0$ with
    \begin{equation}
    	\label{eq:phi-estimate}
        \phi(t) \ge c t, \quad (0 \le t \le M).
    \end{equation}
    Indeed, by the definition of $\phi_0$, there exists $t_0>0$ such that
    $\phi(t) > (\phi_0/2) t$ for $t \in (0, t_0)$.
    Since $\phi$ is increasing, we have $\phi(t) \ge \phi(t_0) \ge (\phi_0/2) t_0$ for $t\ge t_0$.
    This yields $c = \phi_0 t_0/(2M)$.
    
    Assuming that $\norm{u}_{L^\infty(\Omega)} \le M < \infty$,
    we now let $\{u_k\}_{k=1}^\infty \subset \pwc(\Omega)$ 
    approximate $u$ weakly* in $\BVspace(\Omega)$.
    We may assume that
    \begin{equation}
        \label{eq:tvphid-remedy-bound}
        \norm{u^k}_{L^\infty(\Omega)} \le M,
    \end{equation}
    because if this would not hold, then we could truncate $u^k$,
    and the modified sequence $\{u^k_M\}_{k=1}^\infty$ would still
    converge to $u$ weakly* in $\BVspace(\Omega)$ with
    $\widetilde\TVphid(u^k_M) \le \widetilde\TVphid(u^k)$.
    Thanks to \eqref{eq:phi-estimate} and \eqref{eq:tvphid-remedy-bound}, we have
    \[
        c\TV(u^k) \le \widetilde\TVphid(u^k), \quad (k=1,\ldots,\infty).
    \]
    By the lower semicontinuity of $\TV(\cdot)$, we obtain
    \[
        c\TV(u) 
        \le
        \liminf_{k \to \infty} \TV(u^k)
        \le
        \liminf_{k \to \infty} \widetilde\TVphid(u^k).
    \]
    Since the approximating sequence $\{u^k\}_{k=1}^\infty$ 
    was arbitrary, the claim follows.
\end{proof}

\section{Limiting aspects of the continuous $\TVphi$ model}
\label{sec:tvphic}

We now consider the continuous model \eqref{eq:tvqc} or \eqref{eq:tvq-cont-disc}.
Both are common in works aiming to model real image statistics. We initially restrict our 
attention to the following energies $\phi$.

\begin{definition}
    We denote by $\Wc$ the class of increasing, subadditive, 
    continuous functions $\phi: \nonnegR \to \nonnegR$ 
    with $\phi^\infty=0$.
\end{definition}

\begin{example}
    Examples of $\phi \in \Wc$ include in particular $\phi(t)=t^q$ 
    for $q \in (0, 1)$, as well as $\phi(t)=\alpha t/(1+\alpha t)$ and
    $\phi(t)=\log(\alpha t + 1)$ for $\alpha > 0$.
\end{example}

\begin{definition}    
Given an energy $\phi$, we start with the $C^1$ model \eqref{eq:tvqc},
which we now denote by
\[
    \widetilde{\TVphic}(u) \defeq \int_\Omega \phi(|\grad u(x)|) \d x,
    \quad
    (u \in C^1(\Omega)).
\]
In order to extend this to $u \in \BVspace(\Omega)$, we take the weak*
lower semicontinuous envelope
\[
    \TVphic(u) \defeq \liminf_{\substack{u^i \weaktostar u,\\ u^i \in C^1(\Omega)}} \widetilde{\TVphid}(u^i).
\]
In the definition, the convergence is weakly* in $\BVspace(\Omega)$. 
\end{definition}

We emphasise that it is crucial to define $\TVphic$ through this 
limiting process in order to obtain weak* lower semicontinuity. This is 
useful to show the existence of solutions to variational problems 
with the regulariser $\TVphic$ in $\BVspace(\Omega)$ -- or a larger space,
as there is no guarantee that $\TVphic(u) < \infty$ would imply
$u \in \BVspace(\Omega)$. 

Our main result on the $\TVphic$ model states the following theorem.

\begin{theorem}
    \label{thm:tvphic}
    Let $\phi \in \Wc$, and suppose that $\Omega \subset \R^m$ 
    has a Lipschitz boundary. Then
    \[
        \TVphic(u) = 0 \quad \text{ for } \quad u \in \BVspace(\Omega).
    \]
\end{theorem}

The main ingredient of the proof is Lemma \ref{lemma:tvphic}, 
which is provided by a simple result.

\begin{lemma}
    \label{lemma:phi-wc-approx}
    Let $\phi \in \Wc$. Then there exist $a, b >0$ such that
    \[
        \phi(t) \le a + b t, \quad (t \in \nonnegR).
    \]
\end{lemma}
\begin{proof}
    Since $\phi^\infty=0$, 
    we can find $t_0 > 0$ such that $\phi(t)/t \le 1$ for $t \ge t_0$.
    Thus, because $\phi$ is increasing, we have $\phi(t) \le \phi(t_0)+t$ for every $t \in \nonnegR$.
\end{proof}

\begin{lemma}
    \label{lemma:tvphic}
    Let $\phi \in \Wc$, and suppose that $\Omega \subset \R^m$ is bounded with 
    Lipschitz boundary. Then
    \begin{equation}
        \label{eq:tvphic-bound}
        \TVphic(u) \le \int_\Omega \phi(|\grad u(x)|) \d x,
        \quad
        (u \in \BVspace(\Omega)).
    \end{equation}
\end{lemma}

Observe the difference between Lemma \ref{lemma:tvphic} 
and Theorem \ref{thm:tvphid}. The former shows that in the limit
of $\widetilde\TVphic$, the singular part is completely free, 
whereas the latter shows that in the limit of $\widetilde\TVphid$, 
only the jump part is allowed at all!

\begin{proof}
    We may assume that
    \begin{equation}
        \notag
        \int_\Omega \phi(|\grad u(x)|) \d x < \infty,
    \end{equation}
    because otherwise there is nothing to prove.
    We let $u_0 \in \BVspace(\R^m)$ denote the zero-extension
    of $u$ from $\Omega$ to $\R^m$. Then
    \[
        Du_0=Du - \nu_{\BD \Omega} u^- \H^{n-1} \restrict \BD \Omega
    \]
    for $u^-$ the interior trace of $u$ on $\BD \Omega$, and $\nu_{\BD \Omega}$
    the exterior normal of $\Omega$.
    In fact \cite[Section 3.7]{ambrosio2000fbv} there exists a constant $C=C(\Omega)$ such that
    \[
        \norm{\nu_{\BD \Omega} u^- \H^{n-1} \restrict \BD \Omega}_{\Meas(\R^m; \R^m)}
        \le C \norm{u}_{\BVspace(\Omega)}.
    \]
    We pick some $\rho \in C_c^\infty(\R^m)$ with $0 \le \rho \le 1$, 
    $\int \rho \d x=1$, and $\supp \rho \subset \B(0, 1)$.
    We then define the family of mollifiers
    $\rho_\epsilon(x) \defeq \epsilon^{-n} \rho(x/\epsilon)$
    for $\epsilon >0$, and let
    \[
        u_\epsilon \defeq (\rho_\epsilon * u_0)|\Omega.
    \]
    Then $u_\epsilon \in C^\infty(\closure \Omega)$, and 
    $u_\epsilon \to u$ strongly in $L^1(\Omega)$ as $\epsilon\downto0$.
    As $\abs{D u_\epsilon}(\omega) \le \abs{D u_0}(\Omega)$,
    it follows that $u_\epsilon \weaktostar u$ weakly* in
    $\BVspace(\Omega)$; see, e.g., \cite[Proposition 3.13]{ambrosio2000fbv}. Thus
    \[
        \TVphic(u) \le \liminf_{\epsilon \downto 0} \widetilde \TVphic(u_\epsilon).
    \]
    In order to obtain the conclusion of the theorem, 
    we just have to calculate the right hand side.
    
    We have
    \begin{equation}
        \label{eq:grad-u-epsilon-split}
        \begin{split}
        |\grad u_\epsilon(x)|
        &
        = \left|\int_{\R^m} \rho_\epsilon(x-y) \d D u_0(y)\right|
        \le
        \int_{\R^m} \rho_\epsilon(x-y) \d \abs{D u_0}(y)
        \\
        &
        \le
        \int_{\R^m} \rho_\epsilon(x-y) |\grad u_0(y)| \d y
        +
        \int_{\R^m} \rho_\epsilon(x-y) \d \abs{D^s u_0}(y).
        \end{split}
    \end{equation}
    We approximate the terms for the absolutely continuous 
    and singular parts differently.
    Starting with the absolutely continuous part, we let
    $K$ be a compact set such that $\Omega + \B(0, 1) \subset K$,
    and define
    \[
        g_0(x) \defeq |\grad u_0(x)|
        \quad
        \text{and}
        \quad
        g_\epsilon(x) 
        \defeq 
        \int_{\R^m} \rho_\epsilon(x-y) |\grad u_0(y)| \d y.
    \]
    Then $g_\epsilon \to g_0$ in $L^1(K)$,
    and $g_\epsilon|(\R^m \setminus K)=0$ for $\epsilon \in (0, 1)$.
    By the $L^1$ convergence, we can find a sequence $\epsilon^i \downto 0$
    such that $g_{\epsilon^i} \to g_0$ almost uniformly. 
    Consequently, given $\delta > 0$, we may find a set 
    $E \subset K$ with $\L^m(K \setminus E) < \delta$ and
    $g_{\epsilon^i} \to g_0$ uniformly on $E$. 
    We may assume that each $\epsilon^i$ is small enough such that
    \begin{equation}
        \label{eq:k-l1-bound}
        \norm{g_{\epsilon^i} - g_0}_{L^1(K)} \le \delta.
    \end{equation}
    %
    %
    Lemma \ref{lemma:phi-wc-approx} provides for some $a, b > 0$ the estimate
    \begin{equation}
        \label{eq:phi-a-b-approx}
        \phi(t) \le a + bt.
    \end{equation}
    From the uniform convergence on $E$, it follows that for large enough $i$, 
    we have
    \[
        \phi(g_{\epsilon^i}(x)) \le
        \phi(1+g_0(x)) \le v(x) \defeq a + b(1+g_0(x)),
        \quad
        (x \in E).
    \]
    Since $E \subset K$ is bounded, $v \in L^1(E)$.
    The reverse Fatou inequality on $E$ gives the estimate
    \begin{equation}
        \label{eq:tvqc-ac-est1}
        \limsup_{i \to \infty}
        \int_E
        \phi(g_{\epsilon^i}(x)) \d x
        \le
        \int_E
            \limsup_{i \to \infty}
            \phi(g_{\epsilon^i}(x)) \d x
        \le
        \int_E
            \phi(g_0(x))
        \d x.
    \end{equation}
    On $K \setminus E$, we obtain the estimate
    \begin{equation}
        \label{eq:tvqc-ac-est2}
        \begin{aligned}
        \int_{K \setminus E}
        \phi(g_{\epsilon^i}(x)) 
        \d x
        &
        \le
        \int_{K \setminus E}
        \phi(g_0(x)) 
        +
        \phi(\abs{g_{\epsilon^i}(x)-g_0(x)}) 
        \d x
        &&
        \text{(by subadditivity)}
        \\
        &
        \le
        \int_{K \setminus E}
            \phi(g_0(x))
            \d x
        +
        a \L^m(K \setminus E) + b \norm{g_{\epsilon^i}- g_0}_{L^1(K)}
        &&
        \text{(by \eqref{eq:phi-a-b-approx})}
        \\
        &
        \le
        \int_{K \setminus E}
            \phi(g_0(x))
            \d x
            +(a+b)\delta.
        &&
        \text{(by \eqref{eq:k-l1-bound})}
        \end{aligned}
    \end{equation}
    Combining the estimates \eqref{eq:tvqc-ac-est1} and \eqref{eq:tvqc-ac-est2}, we
    have
    \[
        \limsup_{i \to \infty}
        \int_\Omega
        \phi(g_{\epsilon^i}(x)) \d x
        \le
        \int_{K}
            \phi(g_0(x))
            \d x
        +(a+b)\delta.
    \]
    Since $\delta > 0$ was arbitrary, and we may always find an almost uniformly 
    convergent subsequence of any subsequence of $\{g_\epsilon\}_{\epsilon > 0}$, 
    we conclude that
    \begin{equation}
        \label{eq:tvqc-ac-estim}
        \limsup_{\epsilon \downto 0}
        \int_\Omega
        \phi(g_{\epsilon}(x)) \d x
        \le
        \int_{K}
            \phi(|\grad u_0(x)|)
            \d x
        =
        \int_{\Omega}
            \phi(|\grad u(x)|)
            \d x.
    \end{equation}

    Let us then consider the singular part in \eqref{eq:grad-u-epsilon-split}.
    We observe that $\int_{\R^m} \rho_\epsilon(x-y) \d \abs{D^s u_0(y)}=0$,
    for $x \in \R^m \setminus K$. 
    If we define
    \[
        f_\epsilon(x) \defeq \epsilon^{-m} \abs{D^s u_0}(\B(x, \epsilon)),
        \quad (x \in K),
    \]
    then by Fubini's theorem
    \begin{equation}
        \label{eq:f3-ds-u-k}
        \begin{split}
        \int_K f_\epsilon(x) \d x
        =
        \epsilon^{-m} \int_K \int_K \chi_{\B(x, \epsilon)}(y) \d \abs{D^s u_0}(y) \d x
        &
        =
        \epsilon^{-m} \int_K \int_K \chi_{\B(y, \epsilon)}(x) \d x \d \abs{D^s u_0}(y)
        \\
        &
        \le \omega_m \abs{D^s u_0}(K).
        \end{split}
    \end{equation}
    Here $\omega_m$ is the volume of the unit ball in $\R^m$.
    Moreover, by the Besicovitch derivation theorem 
    (discussed, for example, in \cite{ambrosio2000fbv,mattila1999geometry}), we have
    \[
        \lim_{\epsilon \downto 0} f_\epsilon(x) = 0, \quad (\L^m\text{-\ae} x \in K).
    \]
    Because $\L^m(K) < \infty$, Egorov's theorem shows that
    $f_\epsilon \to 0$ almost uniformly. 
    Thus, for any $\delta>0$, there exists a set $K_\delta \subset K$
    with $\L^m(K \setminus K_\delta) \le \delta$ and
    $f_\epsilon \to 0$ uniformly on $K_\delta$. 
    
    Next we study $K \setminus K_\delta$.
    We pick an arbitrary $\sigma > 0$. 
    Because $\phi(t)/t \to 0$
    as $t \to \infty$, there exists $t_0 > 0$ such that
    $\phi(t) \le \sigma t$ for $t \ge t_0$.
    In fact, because $\phi$ is lower semicontinuous and $\phi(0)=0$, 
    if we choose
    \[
        t_0 \defeq \inf\{t \ge 0\mid \phi(t) < \sigma t\},
    \]
    then  $\phi(t_0) = \sigma t_0$. Thus, because $\phi$ is increasing
    \begin{equation}
        \label{eq:phi0}
        \phi(t) \le \tilde \phi(t) 
                \defeq
                \sigma(t_0+t),
        \quad (t \in \nonnegR).
    \end{equation}
    Choosing $\epsilon \in (0, 1)$ such that
    $f_\epsilon \le \delta$ on $K_\delta$,
    and using $\rho_\epsilon \le \epsilon^{-m} \chi_{\B(0, \epsilon)}$,
    we may approximate
    \begin{equation}
        \label{eq:tvqc-singular-estim-1}
        \begin{aligned}
        \int_{\R^m} \phi\left(\int_{\R^m} \rho_\epsilon(x-y) \d \abs{D^s u_0}(y)\right) \d x
        &
        \le
        \int_{\R^m} \phi\left(f_\epsilon(x)\right) \d x
        \\
        &
        \le
        \int_{K_\delta} \phi\left(f_\epsilon(x)\right) \d x
        +
        \int_{K \setminus K_\delta} \tilde \phi\left(f_\epsilon(x)\right) \d x
        &&
        \\
        &
        \le
        \int_{K_\delta} \phi(\delta) \d x
        +
        \int_{K \setminus K_\delta} \sigma\left(t_0 + f_\epsilon(x)\right) \d x
        &&
        \\
        &
        \le
        \L^m(K) \phi(\delta)
        +
        \delta\sigma t_0 + \sigma \omega_m \abs{D^s u_0}(K)
        &
        &
        \text{(by \eqref{eq:f3-ds-u-k}).}
        \end{aligned}
    \end{equation}
    Thus
    \[
        \liminf_{\epsilon \downto 0}
        \int_{\R^m} \phi\left(\int_{\R^m} \rho_\epsilon(x-y) \d \abs{D^s u_0}(y)\right) \d x
        \le
        \L^m(K) \phi(\delta)
        +
        \delta\sigma t_0 + \sigma \omega_m \abs{D^s u_0}(K).
    \]
    Observe that the choices of $\sigma$ and $t_0$ are independent of $\delta$.
    Therefore, because $\delta > 0$ was arbitrary, using the continuity of $\phi$ 
    we deduce that we may set $\delta=0$ above. But then, because $\sigma > 0$ was
    also arbitrary, we deduce
    \begin{equation}
        \label{eq:tvqc-singular-estim}
        \lim_{\epsilon \downto 0}
        \int_{\R^m} \phi\left(\int_{\R^m} \rho_\epsilon(x-y) \d \abs{D^s u_0}(y)\right) \d x
        = 0.
    \end{equation}
    
    Finally, combining the estimate \eqref{eq:tvqc-ac-estim} for the absolutely continuous part
    and the estimate \eqref{eq:tvqc-singular-estim} for the singular part
    with \eqref{eq:grad-u-epsilon-split}, we deduce that
    \[
        \limsup_{\epsilon \downto 0}
        \widetilde \TVphic(u_\epsilon)
        =
        \limsup_{\epsilon \downto 0}
        \int_{\R^m} \phi\left(\int_{\R^m} \rho_\epsilon(x-y) \d \abs{D u_0}(y)\right) \d x 
        \le 
        \int_{\Omega}
            \phi(|\grad u(x)|)
            \d x.
    \]    
    This concludes the proof of \eqref{eq:tvphic-bound}.
\end{proof}

\begin{proof}[Proof of Theorem \ref{thm:tvphic}]
    We employ the bound \eqref{eq:tvphic-bound} of
    Lemma \ref{lemma:tvphic}, but still have to
    extend it to a possibly unbounded domain $\Omega$.
    For this purpose, we let $R > 0$ be arbitrary, and apply
    the lemma to $u_R \defeq u|\B(0, R)$. Then
    \begin{equation*}
        \TVphic(u_R)
        \le \int_\Omega \phi(|\grad u_R(x)|) \d x
        \le \int_\Omega \phi(|\grad u(x)|) \d x.
    \end{equation*}
    But $u_R \weaktostar u$ weakly* in $\BVspace(\Omega)$ as $R \upto \infty$;
    indeed $L^1$ convergence is obvious, and for any 
    $\phi \in C_c^\infty(\Omega; \R^m)$, we have $\support \phi \in \B(0, R)$
    for large enough $R$, so that $Du_R(\phi)=Du(\phi)$.
    Therefore, because $\TVphic$ is weakly* 
    lower semicontinuous by construction, we conclude that
    \begin{equation}
        \label{eq:tvphic-bound-full}
        \TVphic(u) \le \int_\Omega \phi(|\grad u(x)|) \d x.
    \end{equation}

    Given any $u \in C^1(\Omega)$, we may find $u_h \in \pwc(\Omega)$, ($h > 0$),
    strictly convergent to $u$ in $\BVspace(\Omega)$ 
    \cite{ckp1999regularization}. But \eqref{eq:tvphic-bound-full} 
    shows that
    \[
        \TVphic(u_h) = 0.
    \]
    By the weak* lower semicontinuity of $\TVphic$ we conclude
    \[
        \TVphic(u) \le \liminf_{h \downto 0} \TVphic(u_h) = 0,
        \quad
        (u \in C^1(\Omega)).
    \]
    Another referral to lower-semicontinuity now shows that $\TVphic(u)=0$ 
    for any $u \in \BVspace(\Omega)$.
\end{proof}

Similarly to Proposition \ref{prop:tvphid-remedy}
for $\TVphid$, we have the following more positive result.

\begin{proposition}
    \label{prop:tvphic-remedy}
    Let $\phi: \nonnegR \to \nonnegR$ be lower semicontinuous and satisfy $\phi(0)=0$.
    Then the following hold.
    \begin{enumroman}
    \item
    If $\phi_0 < \infty$ and $\phi$ is subadditive, then there exist a constant $C>0$ such that
    \[
        \TVphic(u) \le C\,\TV(u),
        \quad (u \in \BVspace(\Omega)).
    \]
    \item
    \label{item:tvphic-remedy-ii}
    If $\phi_0, \phi^\infty >0$ and $\phi$ is increasing, 
    then there exists also a constant $c > 0$ such that
    \[
        c\,\TV(u) \le \TVphic(u),
        \quad (u \in \BVspace(\Omega)).
    \]
    \end{enumroman}
\end{proposition}

\begin{remark}
    If we assume that $\phi$ is concave, the condition $\phi_0>0$ 
    in \ref{item:tvphic-remedy-ii} follows from the other assumptions.
\end{remark}

\begin{proof}
    The proof of the upper bound follows exactly as the
    upper bound in Proposition \ref{prop:tvphid-remedy},
    just replacing approximation in $\pwc(\Omega)$ by $C^1(\Omega)$.
    
    For the lower bound, first of all, we observe that 
    there exists $t^\infty>0$ such that
    $\phi(t) \ge (\phi^\infty/2) t$, ($t \ge t^\infty$).
    Secondly, there exists $t_0>0$ such that
    $\phi(t) \ge (\phi_0/2) t$, ($0 \le t \le t_0$).
    Since $\phi$ is increasing,
    $\phi(t) \ge \phi(t_0) \ge t \phi(t_0)/t_0$, ($t_0 \le t \le t_\infty$).
    Consequently
    \[
        \phi(t) \ge c t, \quad (t\ge 0),
        \quad
        \text{for } c \defeq \min\{\phi^\infty/2, \phi(t_0)/t_\infty, \phi^0/2\}.
    \]
    Therefore
    \[
        c\TV(u') \le \widetilde\TVphic(u'), \quad (u \in C^1(\Omega)).
    \]
    The claim now follows from the weak* lower semicontinuity of $\TV$
    as in the proof of Proposition \ref{prop:tvphid-remedy}.
\end{proof}

In fact, in most of the interesting cases we may prove a slightly stronger result.

\begin{theorem}
    \label{thm:tvphic-lin}
    Let $\phi: \nonnegR \to \nonnegR$ be concave 
    with $\phi(0)=0$ and $0 < \phi^\infty < \infty$.
    Suppose that $\Omega \subset \R^m$ has a Lipschitz boundary. Then
    \begin{equation}
        \label{eq:tvphic-form}
        \TVphic(u) = \phi^\infty\TV(u).
        \quad
        (u \in \BVspace(\Omega)).
    \end{equation}
\end{theorem}

\begin{proof}
    We first suppose that $\Omega$ is bounded.
    The proof of the upper bound
    \begin{equation}
        \label{eq:tvphic-lin-up}
        \TVphic(u) \le \int_\Omega \phi(|\grad u(x)|) \d x + \phi^\infty\abs{D^s u}(\Omega),
    \end{equation}
    is then a modification of Lemma \ref{lemma:tvphic}.
    The estimate \eqref{eq:tvqc-ac-estim} for the absolutely continuous part
    follows as before.
    For the singular part, we observe that \eqref{eq:phi0} holds for any $\sigma > \phi^\infty$. 
    Therefore, proceeding as before, we obtain in place of
    \eqref{eq:tvqc-singular-estim} the estimate
    \begin{equation}
        \label{eq:tvqc-lin-singular-estim}
        \lim_{\epsilon \downto 0}
        \int_{\R^m} \phi\left(\int_{\R^m} \rho_\epsilon(x-y) \d \abs{D^s u_0}(y)\right) \d x
        \le \sigma\abs{D^s u}(\Omega).
    \end{equation}
    Letting $\sigma \downto \phi^\infty$ and combining \eqref{eq:tvqc-ac-estim} 
    with \eqref{eq:tvqc-lin-singular-estim} we get \eqref{eq:tvphic-lin-up}.
    As in Theorem \ref{prop:tvphic-remedy}, we may extend this bound to a
    possibly unbounded $\Omega$.

    If $u \in C^1(\Omega)$, we may again approximate $u$ strictly in $\BVspace(\Omega)$ 
    by piecewise constant functions $\{u^i\}_{i=1}^\infty$. 
    By the lower semicontinuity of $\TVphic$ and \eqref{eq:tvphic-lin-up}, we then have
    \begin{equation}
        \label{eq:tvphic-linfty-up}
        \TVphic(u)
        \le
        \liminf_{i \to \infty} \phi^\infty \abs{D^s u}(\Omega)=\phi^\infty \abs{Du}(\Omega).
    \end{equation}
    Finally, we observe that by concavity
    \[
        \phi(t) \ge \phi^\infty t.
    \]
    Thus $\widetilde \TVphic(u) \ge \phi^\infty \abs{Du}(\Omega)$. 
    We immediately obtain \eqref{eq:tvphic-form} for $u \in C^1(\Omega)$.
    By strictly convergent approximation, we then extend the result to $u \in \BVspace(\Omega)$.
\end{proof}

\section{Discussion}
\label{sec:discuss}

Theorem \ref{thm:tvphic} and Theorem \ref{thm:tvphic-lin} show that we cannot
hope to have a simple weakly* lower semicontinuous non-convex total variation model 
as a prior for image gradient distributions. In fact, it follows from \cite{bouchitte1990new}, 
see also \cite[Section 5.1]{ambrosio2000fbv} and \cite[Theorem 5.14]{fonseca2007mmc}, 
that lower semicontinuity of the continuous $\TVphic$ model is only possible for 
convex $\phi$. The problem is: if $\phi^\infty$ is less than $\phi'(t)$, then 
image edges are always cheaper than smooth transitions. 
If $\phi^\infty=0$, they are so cheap that  we get a zero functional 
at the limit for a general class of functions. 
If $\phi^\infty > 0$ and $\phi$ is concave, then we get a factor of $\TV$ as result.
If $\phi$ is not concave, we still have the upper bound \eqref{eq:tvphic-linfty-up};
it may however be possible that some gradients are cheaper than jumps. This would in
particular be the case with Huber regularisation of $\phi(t)=t$. More about the jump
set of solutions to Huber-regularised as well as non-convex total variation
models may be read in \cite{tuomov-jumpset}.

In fact, in \cite{HiWu13_siims} Huber regularisation 
was used with $\phi(t)=t^q$ for $q \in (0, 1)$ for algorithmic reasons.
For small $\gamma>0$, this is defined as
\begin{equation}
    \label{eq:huber}
    \widetilde \phi(t) \defeq
    \begin{cases}
        \frac{1}{q}t^q-\frac{2-q}{2q}\gamma^q, & t > \gamma, \\
        \frac{1}{2}\gamma^{q-2}t^2, & t \in [0, \gamma].
    \end{cases}
\end{equation}
Then $\widetilde \phi(t) \le \phi(t)$, so that
\[
    \TVphic[\widetilde\phi] \le \TVphi[\phi] = 0.
\]
Therefore Huber regularisation provides no remedy in this case.

In contrast to the continuous $\TVphic$ model, according to Theorem \ref{thm:tvphid}, 
the discrete model works correctly for $\phi(t)=t^q$ and generally
$\phi \in \Wd$, if the desire is to force piecewise constant solutions 
to \eqref{eq:tvq-rof}. As we saw in the comments preceding
Proposition \ref{prop:tvphid-remedy}, it however does not
force piecewise constant solutions for some of the energies
$\phi$ typically employed in this context. Generally,
what causes piecewise constant solutions is the property
$\phi^0=\infty$. If one does
not desire piecewise constant solutions, one can therefore
use Huber regularisation or linearise $\phi$ for $t < \delta$.
The latter employs
\[
    \widetilde \phi(t)=\begin{cases}
                \phi(t)-\phi(\delta)+\phi'(\delta)\delta, & t > \delta, \\
                \phi'(\delta)t, & t \le \delta.
            \end{cases}
\]
Then $\phi(t) \le C t$ for some $C > 0$, so that $\TVphid(u) < \infty$ for
every $u \in \BVspace(\Omega)$. 
We also note that although this approach
defines a regularisation functional on all of $\BVspace(\Omega)$, it
cannot be used for modelling the distribution of gradients in
real images, the purpose of the $\TVphic$ model. 
In fact, as in the the $\TVphid$ model we cannot control the penalisation of $\grad u$ 
beyond a constant factor.


In summary, the $\TVphid$ model works as intended for $\phi \in \Wd$
-- it enforces piecewise constant solutions. The $\TVphic$ model however
is not theoretically sound in function spaces.
We will therefore next seek ways to fix it.

\section{Multiscale regularisation and area-strict convergence}
\label{sec:areastrict}

The problem with the $\TVphic$ model is that weak* lower semicontinuity
is too strong a requirement. We need a weaker type of lower semicontinuity,
or, in other words, a stronger type of convergence. Norm convergence in
BV is too strong; it would not be possible at all to approximate edges. Strict convergence
is also still too weak, as can be seen from the proof of Lemma \ref{lemma:tvphic}.
Strong convergence in $L^2$, which we could in fact obtain from
strict convergence for $\Omega \subset \R^2$ (see \cite{lions1985concentration,rindler2013strictly}), is also not enough, as a stronger form of gradient convergence is the important part.
A suitable mode of convergence is the so-called \emph{area-strict convergence}
\cite{delladio1991lower,kristensen2010relaxation}. For our purposes,
the following definition is the most appropriate one.

\begin{definition}
    Suppose $\Omega \subset \R^n$ with $n \ge 2$.
    The sequence $\{u^i\}_{i=1}^\infty \subset \BVspace(\Omega)$ converges
    to $u \in \BVspace(\Omega)$ \emph{area-strictly} if the sequence $\{U_i\}_{i=1}^\infty$ with
    $U^i(x) \defeq (x/|x|, u^i(x))$ converges strictly in $\BVspace(\Omega; \R^{n+1})$
    to $U(x) \defeq (x/|x|, u(x))$.
\end{definition}

In other words, $\{u^i\}_{i=1}^\infty$ converges to $u$ area-strictly if
$u^i \to u$ strongly in $L^1(\Omega)$, 
$Du^i \weaktostar Du$ weakly* in $\Meas(\Omega; \R^n)$, 
and $\mathcal{A}(u^i) \to \mathcal{A}(u)$ for the \term{area functional}
\[
    \mathcal{A}(u) \defeq
    \int_\Omega \sqrt{1+|\grad u(x)|^2} \d x
    + \abs{D^s u}(\Omega).
\]
Here we recall that $D^s u$ is the singular part of $D u$.
It can be shown that area-strict convergence is stronger than strict convergence,
but weaker than norm convergence. 

In order to state a continuity result with respect to area-strict
convergence, we need a few definitions.
Specifically, we denote the Sobolev conjugate
\[
    1^* \defeq \begin{cases}
        n/(n-1), & n>1, \\
        \infty, & n=1,
    \end{cases}
\]
and define
\[
    u^\theta(x) \defeq
    \begin{cases}
        \theta u^+(x) + (1-\theta)u^-(x), & x \in J_u, \\
        \tilde u(x), & x \not\in S_u.
    \end{cases}
\]

In \cite{rindler2013strictly}, see also \cite{kristensen2010relaxation}, 
the following result is proved.

\begin{theorem}
    Let $\Omega$ be a bounded domain with Lipschitz boundary,
    $p \in [1, 1^*]$ if $n\ge 2$ and $p \in [1, 1^*)$ if $n = 1$.
    Let $f \in C(\Omega \times \R \times \R^n)$ satisfy
    \[
        \abs{f(x, y, A)} \le C(1+\abs{y}^p + \abs{A}),
        \quad
        ((x, y, A) \in \Omega \times \R \times \R^n),
    \]
    and assume the existence of $f^\infty \in C(\Omega \times \R \times \R^n)$,
    defined by
    \[
        f^\infty(x, y, A) \defeq \lim_{\substack{x' \to x\\y' \to y\\ A' \to A\\ t \to \infty}} \frac{f(x', y', tA')}{t}.
    \]
    Then the functional 
    \[
        \mathcal{F}(u) \defeq \int_\Omega f(x, u(x), \grad u(x)) \d x
        + \int_\Omega \int_0^1 f^\infty(x, u^\theta(x), \frac{d D^s u}{d\abs{D^s u}} (x)) \d \abs{D^s u}(x).
    \]
    is area-strictly continuous on $\BVspace(\Omega)$.
\end{theorem}

Applied to non-convex total variation, we immediately obtain the following.

\begin{corollary}
    \label{cor:tvphics-areastrict}
    Suppose $\phi \in C(\nonnegR)$, $\phi^\infty$ exists, and $\phi(t) \le C(1+ t)$, ($t \in \nonnegR$).
    Then the functional
    \[
        \TVphics(u) \defeq
        \int_\Omega \phi(|\grad u(x)|) \d x
        + \phi^\infty \abs{D^s u}(\Omega),
        \quad
        (u \in \BVspace(\Omega)),
    \]
    is area-strictly continuous on $\BVspace(\Omega)$.
\end{corollary}

But how could we obtain area-strict convergence of an infimising sequence
of a variational problem? In \cite{tuomov-bd,tuomov-ap1} 
the following multiscale analysis functional $\eta$ was introduced
for scalar-valued measures $\mu \in \Meas(\Omega)$. 
Given $\eta_0>0$ and $\{\rho_\epsilon\}_{\epsilon>0}$,
a family of mollifiers satisfying the semigroup property
$\rho_{\epsilon+\delta} =\rho_\epsilon * \rho_\delta$,
$\eta$ can be defined as
\[
    \eta(\mu) \defeq \eta_0 \sum_{\ell=1}^\infty \int_{\R^n} (\abs{\mu} * \rho_{2^{-i}})(x) - \abs{\mu * \rho_{2^{-i}}}(x) \d x,
    \quad
    (\mu \in \Meas(\Omega)).
\]
If the sequence of measures
$\{\mu^i\}_{i=1}^\infty \subset \Meas(\Omega)$ satisfies 
$\sup_i \eta(\mu) < \infty$ and $\mu^i \weaktostar \mu$ 
weakly* in $\Meas(\Omega)$, then we have
$\abs{\mu^i}(\Omega) \to \abs{\mu}(\Omega)$. In essence,
the functional $\eta$ penalises the type of complexity of 
measures such as two approaching $\delta$-spikes of different sign,
which prohibits strict convergence.
In Appendix \ref{app:vecteta}, we extend the strict
convergence results of \cite{tuomov-bd,tuomov-ap1} 
to vector-valued $\mu \in \Meas(\Omega; \R^N)$,
in particular the case $\mu=DU$ for $U$ the lifting
of $u$ as discussed above.

In order to bound in $\BVspace(\Omega)$ an infimising sequence 
of problems using $\TVphics$ as a regulariser,
we require slightly stricter assumptions on $\phi$.
These can usually, and particularly in the interesting case
$\phi(t)=t^q$, be easily satisfied by linearising $\phi$ 
above a cut-off point $M$ with respect to the function value. This will force $\phi^\infty > 0$,
which is not required for continuity with respect to 
area-strict convergence in its own right. We will later see that
such a cut-off can be justified by real gradient 
distributions and also argued in numerical experiments.

\begin{definition}
    We denote by $\Was$ the set of functions
    $\phi \in C(\nonnegR)$ such that $\phi^\infty$ exists, and
    for some $c, C > 0$ and $b \ge 0$ the following estimates hold true:
    \[
        c t - b \le \phi(t) \le C(1+ t), \quad (t \in \nonnegR).
    \]
\end{definition}

Now we may prove the following result, which shows that area-strict convergence
and the multiscale analysis functional $\eta$ provide a remedy for the
theoretical difficulties associated with the $\TVphic$ model.

\begin{theorem}
    \label{thm:tvphics}
    Suppose $\Omega \subset \R^n$ is bounded with Lipschitz boundary,
    and $\phi \in \Was$.
    Define $U(x) \defeq (1, u(x))$. Then the functional
    \[
        F(u) \defeq \TVphics(u) + \eta(DU)
    \]
    is weak* lower semicontinuous on $\BVspace(\Omega)$, and any sequence
    $\{u^i\}_{i=1}^\infty \subset L^1(\Omega)$ with
    \[
        \sup_i F(u^i) < \infty
    \]
    admits an area-strictly convergent subsequence.
\end{theorem}

\begin{proof}
    Suppose $\{u^i\}_{i=1}^\infty$ converges weakly* to $u \in \BVspace(\Omega)$.
    Then it follows that $\{U^i\}_{i=1}^\infty$ converge weakly* to $U \in \BVspace(\Omega; \R^{m+1})$.
    If $\liminf_{i \to \infty} \eta(DU^i) = \infty$, we clearly have lower semicontinuity of $F$.
    By switching to an unrelabelled subsequence, we may therefore assume that
    $\sup_i \eta(DU^i) < \infty$. It follows from Theorem \ref{thm:eta} in 
    the Appendix that
    $\abs{DU^i}(\Omega) \to \abs{DU}(\Omega)$. In other words, $\{u^i\}_{i=1}^\infty$ 
    converges area-strictly to $u$.  Applying Corollary \ref{cor:tvphics-areastrict}
    and the weak* lower semicontinuity of $\eta$, we now see that
    \[
        F(u) \le \liminf_{i \to \infty} F(u^i).
    \]
    Thus weak* lower semicontinuity holds true.
    
    Next suppose that $\{u^i\}_{i=1}^\infty \subset L^1(\Omega)$ with
    $\sup_i F(u^i) < \infty$. Since $ct-b \le \phi(t)$ and $\Omega$ is bounded, it follows that
    $\sup_i \TV(u^i) < \infty$. The sequence therefore admits a subsequence, unrelabelled without loss of generality,
    which converges weakly* to some $u \in \BVspace(\Omega)$. Hence, the fact that
    $\{u^i\}_{i=1}^\infty$ admits an area-strictly convergent subsequence now 
    follows as in the previous paragraph.
\end{proof}

%
We immediately deduce the following corollary.

\begin{corollary}
    \label{cor:tvphics-gmin}
    Suppose $\Omega \subset \R^n$ is bounded with Lipschitz boundary,
    $\phi \in \Was$, $J: \BVspace(\Omega) \to \R$ is convex, 
    proper, and weakly* lower semicontinuous, and $J$ satisfies for some $C>0$
    the coercivity condition
    \[
        J(u) \ge C(\norm{u}_{L^1(\Omega)}-1).
    \]
    Then the functional
    \begin{equation}
        \label{eq:areastrict-model}
        G(u) \defeq J(u) + \alpha \TVphics(u) + \eta(DU),
        \quad
        (u \in \BVspace(\Omega)),
    \end{equation}
    admits a minimiser $u \in \BVspace(\Omega)$.
\end{corollary}

\begin{remark}
We can, for example, take $J(u) = \frac{1}{2}\norm{z-u}_{L^2(\Omega)}^2$.
\end{remark}

Observe that
\[
    \eta(DU) = \eta_0 \sum_{\ell=1}^\infty \eta_\ell(DU),
\]
where, for $\epsilon_\ell>0$,
\[
    \begin{split}
    \eta_\ell(DU) 
    & \defeq 
        \int_{\R^n} 
        (\rho_{\epsilon_\ell} * \abs{DU})(x)
         - \abs{\rho_{\epsilon_\ell} * DU}(x) \d x
         \\
         &
         =
         \abs{D^s u}(\Omega)
         +
        \int_{\R^n} 
        \sqrt{1+ \abs{\grad u(x)}^2}
         - \sqrt{1+\abs{(\rho_{\epsilon_\ell} * Du)(x)}^2} \d x.
    \end{split}
\]
In particular, if $u \in W^{1,1}(\Omega)$, then we obtain
with $\grad_\epsilon u \defeq \rho_\epsilon * \grad u$ the 
expression
\[
    \eta_\ell(DU) 
    =
        \int_{\R^n} 
        \sqrt{1+ \abs{\grad u(x)}^2}
         - \sqrt{1+\abs{\grad_{\epsilon_\ell} u(x)}^2} \d x
\]
and the estimate
\[
    \eta_\ell(DU) 
    \le
    \int_{\R^n} \sqrt{\bigl|\abs{\grad u(x)}^2-\abs{\grad_{\epsilon_\ell} u(x)}^2\bigr|} \d x.
\]


The following proposition shows that, in infimising sequences, we may
ignore terms from $\eta$. This justifies the associated numerical approximation.

\begin{proposition}
    \label{prop:partial-eta-sum-limit}
    Suppose $\Omega \subset \R^n$ is bounded with Lipschitz boundary,
    $\phi \in \Was$, and that $J: \BVspace(\Omega) \to \R$ is as in
    Corollary \ref{cor:tvphics-gmin}.
    Let $K_i \in \N^+$ and $\epsilon_i > 0$, $i=1,2,3,\ldots$, satisfy
    \[
        \lim_{i \to\infty} K_i = \infty,
        \quad
        \text{and}
        \quad
        \lim_{i \to\infty} \epsilon_i = 0.
    \]
    Suppose further that $\{u^i\}_{i=1}^\infty \subset \BVspace(\Omega)$ satisfies
    \[
        J(u^i) + \alpha \TVphics(u^i) + \sum_{\ell=1}^{K_i} \eta_\ell(DU^i)
        \le \inf_{u \in \BVspace(\Omega)} G(u) + \epsilon_i,
        \quad
        (i=1,2,3,\ldots).
    \]
    Then we can find $\hat u \in \BVspace(\Omega)$ and a subsequence
    of $\{u^i\}_{i=1}^\infty$, unrelabelled, 
    such that $u^i \to \hat u$ area-strictly, and $\hat u$ minimises $G$.
\end{proposition}

\begin{proof}
    Let $L \defeq \inf_{u \in \BVspace(\Omega)} G(u)$.
    Since there is nothing to prove if $L = \infty$,
    we may assume $L < \infty$. Then we have
    \[
        c\TV(u^i) - b \L^n(\Omega) \le \TVphics(u^i).
    \]
    This yields
    \[
        J(u^i) + \alpha c \TV(u^i) \le \alpha b \L^n(\Omega) + L + \epsilon_i.
    \]
    It follows for a subsequence, unrelabelled,
    that $u^i \weaktostar \hat u$ weakly* for some
    $\hat u \in \BVspace(\Omega)$.
    By the weak* lower semicontinuity of $\eta_\ell$, see
    Theorem \ref{thm:eta}, we then have
    \[
        \sum_{\ell=1}^{K_j} \eta_\ell(D\hat U)
        \le \liminf_{i \to\infty} \sum_{\ell=1}^{K_j} \eta(DU^i)
        \le L,
        \quad
        (j=1,2,3,\ldots).
    \]
    It follows that
    \[
        \eta(D\hat U) = \sum_{\ell=1}^\infty \eta_\ell(D\hat U) \le L.
    \]
    Using Lemma \ref{lemma:eta-ki} with $\mu^i=DU^i$ and $\mu=DU$,
    we see that $u^i \to \hat u$ area-strictly, and that
    \[
        u \mapsto
            J(u) + \alpha \TVphics(u) + \sum_{\ell=1}^{K_j} \eta_\ell(DU)
    \]
    is area-strictly lower semicontinuous for
    for each fixed $j = 1,2,3,\ldots$. This shows that
    $G(\hat u) \le L$. As a consequence, $\hat u$ minimises $G$.
\end{proof}

\section{Remarks on alternative remedies}
\label{sec:alternative}

We now discuss two alternative approaches to make the $\TVphic$ model work in
the limit. These are based on compactifying the differential operator and 
on working in $\SBVspace(\Omega)$, respectively. As we only intend to
demonstrate alternative possibilities, we stay brief here. Hence the proofs have been placed
in the appendix.

\begin{remark}[Compact operators]
Area-strict convergence is not the only possibility to make the $\TVphic$ model function; 
another way to understand the problems with the basic $\TVphic$ model is that 
the operator $\grad$ is not compact. One way to obtain 
a compact operator is by convolution. This is the contents of the following
result.

\begin{theorem}
    \label{thm:compact}
    Let $\{\rho_\epsilon\}_{\epsilon >0}$ be a family of mollifiers,
    $\Omega \subset \R^n$ open,
    and $\phi: \nonnegR \to \nonnegR$ increasing, subadditive and
    continuous with $\phi(0)=0$.
    Fix $\epsilon >0$, and define $D_\epsilon: L^1(\Omega) \to L^1(\R^n; \R^n)$
    by
    \[
        D_\epsilon u \defeq \rho_\epsilon * Du.
    \]
    Then
    \[
        \TVphic[\phi,\epsilon](u) \defeq \int_{\R^n} \phi(|D_\epsilon u(x)|) \d x,
        \quad (u \in \BVspace(\Omega)),
    \]   
    is lower semicontinuous with respect
    to weak* convergence in $\BVspace(\Omega)$.
    Moreover
    \begin{equation}
        \label{eq:tvphic-epsilon-conv}
        \lim_{\epsilon \downto 0} \TVphic[\phi,\epsilon](u) = \widetilde \TVphic(u),
        \quad (u \in C^1(\Omega)).
    \end{equation}
\end{theorem}

We relegate the proof of this theorem to Appendix \ref{app:thm:compact}.
It should be noted that any $u \in L^1(\Omega)$ satisfies
$\TVphic[\phi,\epsilon](u) < \infty$. In particular
\[
    \sup_i \frac{1}{2}\norm{z-u^i}_{L^2(\Omega)}^2 + \alpha \TVphic[\phi,\epsilon](u^i) < \infty
\]
does not guarantee weak* convergence of a subsequence.
For that, an additional $\TV(u^i)$ term (with small factor) 
is required in a $\TVphic[\phi,\epsilon]$ based variational model in image processing. 
\end{remark}

\begin{remark}[The space $\SBVspace(\Omega)$ and $\eta$]
If we apply the $\eta$ functional of \cite{tuomov-bd,tuomov-ap1}
to a bounded sequence of functions $g^i \in L^1(\Omega; \R^m)$, 
then we get strict convergence in this space. It remains to find whether we get
convergence. Then we could regularise $\grad u$ this way, and,
working in the space $\SBVspace(\Omega)$, penalise the jump part 
separately. It turns out that this is possible if we state the
modification $\bar\eta$ of $\eta$ in $L^p(\R^n; \R^m)$ for $p \in (1, \infty)$.
Then strict convergence is equivalent to strong convergence.

With $\epsilon_\ell \downto 0$, $\eta_0 > 0$, and $p \in (1, \infty)$, we define
\begin{equation}
    \label{eq:bareta}
    \bar \eta(g) \defeq \eta_0 \sum_{\ell=1}^\infty \bar \eta_\ell(g),
    \quad 
    \bar \eta_\ell(g) 
    \defeq 
    \norm{g}_{L^p(\R^n; \R^m)} - \norm{g * \rho_{\epsilon_\ell}}_{L^p(\R^n; \R^m)},
    \quad
    (g \in L^p(\Omega; \R^m)).
\end{equation}
Then we have the following result, whose proof is relegated to
Appendix \ref{app:thm:sbveta}.

\begin{theorem}
    \label{thm:sbveta}
    Let $\Omega \subset \R^m$ be open and bounded.
    Suppose $\psi \in \Wd$, and $\phi: \nonnegR \to \nonnegR$ is lower 
    semicontinuous and increasing with $\phi^\infty=\infty$ and
    \begin{equation}
        \label{eq:sbveta-p-bound}
        \norm{g}_{L^p(\Omega; \R^m)}
        \le C\left(1+ \int \phi(|g(x)|) \d x\right),
        \quad (g \in L^p(\Omega; \R^m)),
    \end{equation}
    for some $C>0$, where $p \in (1, \infty)$ is as in \eqref{eq:bareta}. Let
    \[
        F(u)
        \defeq
        \int_{\R^n} \phi(|\grad u(x)|) \d x +
        \bar \eta(\grad u) + \int_{J_u} \psi(\theta_u(x)) \d \H^{m-1}(x).
    \]
    Then $F$ is lower semicontinuous with respect to weak* convergence
    in $\BVspace(\Omega)$.
    In fact, any sequence $\{u^i\}_{i=1}^\infty$ with $\sup_i F(u^i)<\infty$
    admits a subsequence convergent weakly* and in the sense
    \eqref{eq:sbv-conv-1}--\eqref{eq:sbv-conv-3}.
\end{theorem}

\end{remark}

\section{Image statistics and the jump part}
\label{sec:stat}

Our studies in the proceding sections have pointed us to the following question:
Are the statistics of \cite{huang1999statistics} valid when we split
the image into smooth and jump parts? What are the statistics for 
jump heights, and does splitting the gradient into these two parts
alter the distribution for the absolutely continuous part? 
When calculating statistics from discrete images, we do not have
the excuse that the jumps would be negligible, i.e.~$\L^m(J_u)=0$! 

In order to gain some insight, here we did a few experiments with real
photographs, displayed in Figures \ref{fig:dock}--\ref{fig:summer}.
These three photographs represent images with different types of statistics.
The pier photo of Figure \ref{fig:dock} is very simple, with large smooth areas
and some fine structures. The parrot test image in Figure \ref{fig:parrot}
has a good balance of features. The summer lake scene in Figure \ref{fig:summer}
is somewhat more complex, with plenty of fine features.

We split the pixels of each image into edge and smooth parts by a simple 
threshold on the norm $\norm{\grad u(k)}$ of the discrete gradient at each
pixel $k$. Then we find optimal $\alpha$ and $q \in (0, 2)$ for the distribution
\[
    P_t(t) \defeq C_t \exp(-\alpha \phi(t)),
\]    
to match the experimental distribution. This in turn gives rise to the prior
\[
    P_u(u)=C_u \exp\left(-\alpha \int_\Omega \phi(|\grad u(x)|) \d x\right).
\]
Both $C_t, C_u > 0$ are normalising factors.
In practise we do the fitting of $P_t$ to the experimental distribution
by a simple least squares fit on the logarithms of
the distributions. We will comment on the suitability of this approach
later on in this section. In the least squares fit we keep $C$ as a free (unnormalised) 
parameter, and recalculate it after the fit. Observe that the normalisation constant
does not affect the denoising problem (here in the finite-dimensional setting)
\[
    \max_f P_{u|f}(u|f)P_u(u) \propto \max_u \exp\left(-\frac{\sigma^2}{2}\norm{z-u}_{2}^2 - \alpha \widetilde\TVphic(u)\right).
\]
Here we have the Gaussian noise distribution
\[
    P_{f|u}(f|u)=C' \exp\left(-\frac{\sigma^2}{2}\norm{z-u}_{2}^2\right),
\]
for $\sigma$ the noise level.
This gives the statistical interpretation of the denoising model, that of a maximum a posteriori (MAP) estimate.

Finally, in the matter of statistics, we note that the $\TVphi$ prior 
only attempts to correctly model gradient statistics; the modelling of 
histogram statistics with Wasserstein distances was recently studied 
in \cite{rabin2011wasserstein,swoboda2013histogram} together with 
the conventional $\TV$ gradient prior.
It is also worth remarking that our approach of improving the prior based on the statistics of the ground-truth is different from recent approaches that optimise the prior based on the denoising result \cite{delosreyes2014learning,haber2003learning,haber2008numerical,biegler2011large,bui2008model}. These approaches can provide improved results in practise, but no longer have the simple MAP interpretation. It is definitely possible to optimise the parameters of the $\TV^\phi$ model in this manner, but outside the scope of the present already long manuscript.

\newlength{\w}
\newlength{\ws}

\def\minione#1{#1}
\def\minitwo#1{#1}
\setlength{\w}{0.49\textwidth}
\setlength{\ws}{0.32\textwidth}

\newcommand{\statimg}[2]{
    \begin{figure}[!ht]
        \input{img/#1-data.tex}
        \centering
        \minione{
        \begin{subfigure}[t]{\w}
            \centering
            \includegraphics[width=\w,keepaspectratio]{img/#1-orig}
            \caption{#2}
            \label{fig:#1:orig}
        \end{subfigure}
        \begin{subfigure}[t]{\w}
            \centering
            \includegraphics[width=\w,keepaspectratio]{img/#1-edge}
            \caption{Detected edge pixels (red)}
            \label{fig:#1:edge}
        \end{subfigure}
        }
        \minitwo{
        \begin{subfigure}[t]{\ws}
            \centering
            \includegraphics[width=\ws,keepaspectratio]{img/#1-fit-full}
            \caption{Log-histogram and $t^q$ model}
            \label{fig:#1:fit}
        \end{subfigure}
        \begin{subfigure}[t]{\ws}
            \centering
            \includegraphics[width=\ws,keepaspectratio]{img/#1-fit-part}
            \caption{Separate $t^q$ model for edge part}
            \label{fig:#1:fitsplit}
        \end{subfigure}
        \begin{subfigure}[t]{\ws}
            \centering
            \includegraphics[width=\ws,keepaspectratio]{img/#1-fit-opt}
            \caption{Linearised models}
            \label{fig:#1:fitlin}
        \end{subfigure}
        }
        \caption{#2: discrete gradient histogram and least squares models.
            The image intensity in (\subref{fig:#1:orig}) is in the range $[0,255]$, 
            and we have chosen pixels $k$ with $|\grad u(k)| \ge \IMGthres$
            as edges (\subref{fig:#1:edge}).
            The log-histogram of $|\grad u(k)|$ with
            optimal fit of $t \mapsto t^q$ is displayed in (\subref{fig:#1:fit}).
            This is done separately for the edge pixels
            in (\subref{fig:#1:fitsplit}).
            The linearised model is fitted in (\subref{fig:#1:fitlin})
            for the cut-off point $M=\IMGthres$ (manual edge detection),
            $M=\IMGlinoptthres$ (optimal least squares fit).
            We moreover show the empirically best linearised model.
            }
        \label{fig:#1}
    \end{figure}
}

    \begin{figure}[!ht]
        \input{img/dock-data.tex}
        \centering
        \minione{
        \begin{subfigure}[t]{\w}
            \centering
            \includegraphics[width=\w,keepaspectratio]{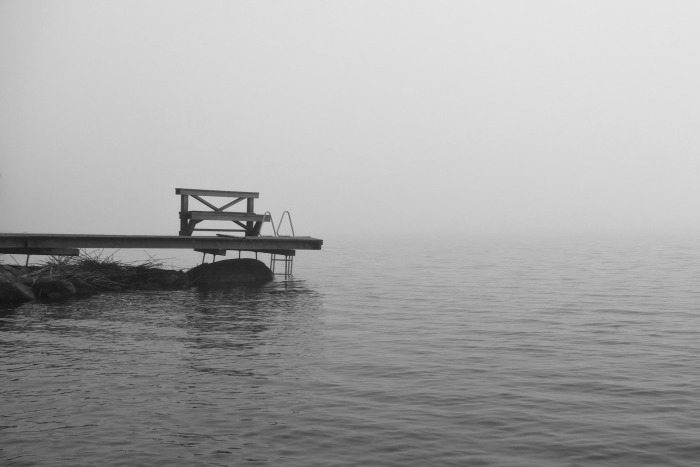}
            \caption{Pier photo}
            \label{fig:dock:orig}
        \end{subfigure}
        \begin{subfigure}[t]{\w}
            \centering
            \includegraphics[width=\w,keepaspectratio]{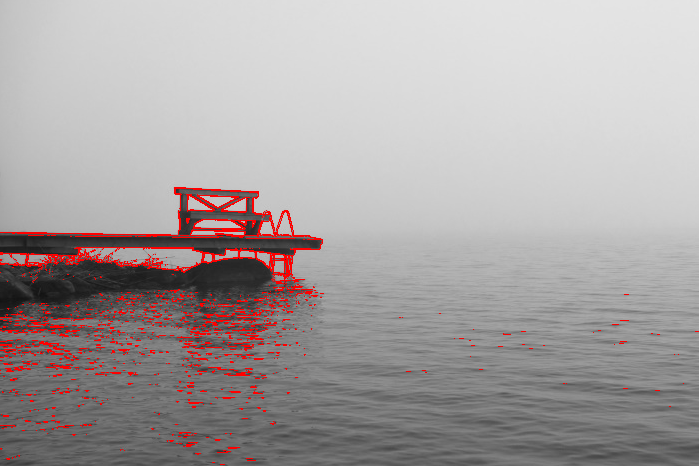}
            \caption{Detected edge pixels (red)}
            \label{fig:dock:edge}
        \end{subfigure}
        }
        \minitwo{
        \begin{subfigure}[t]{\ws}
            \centering
            \includegraphics[width=\ws,keepaspectratio]{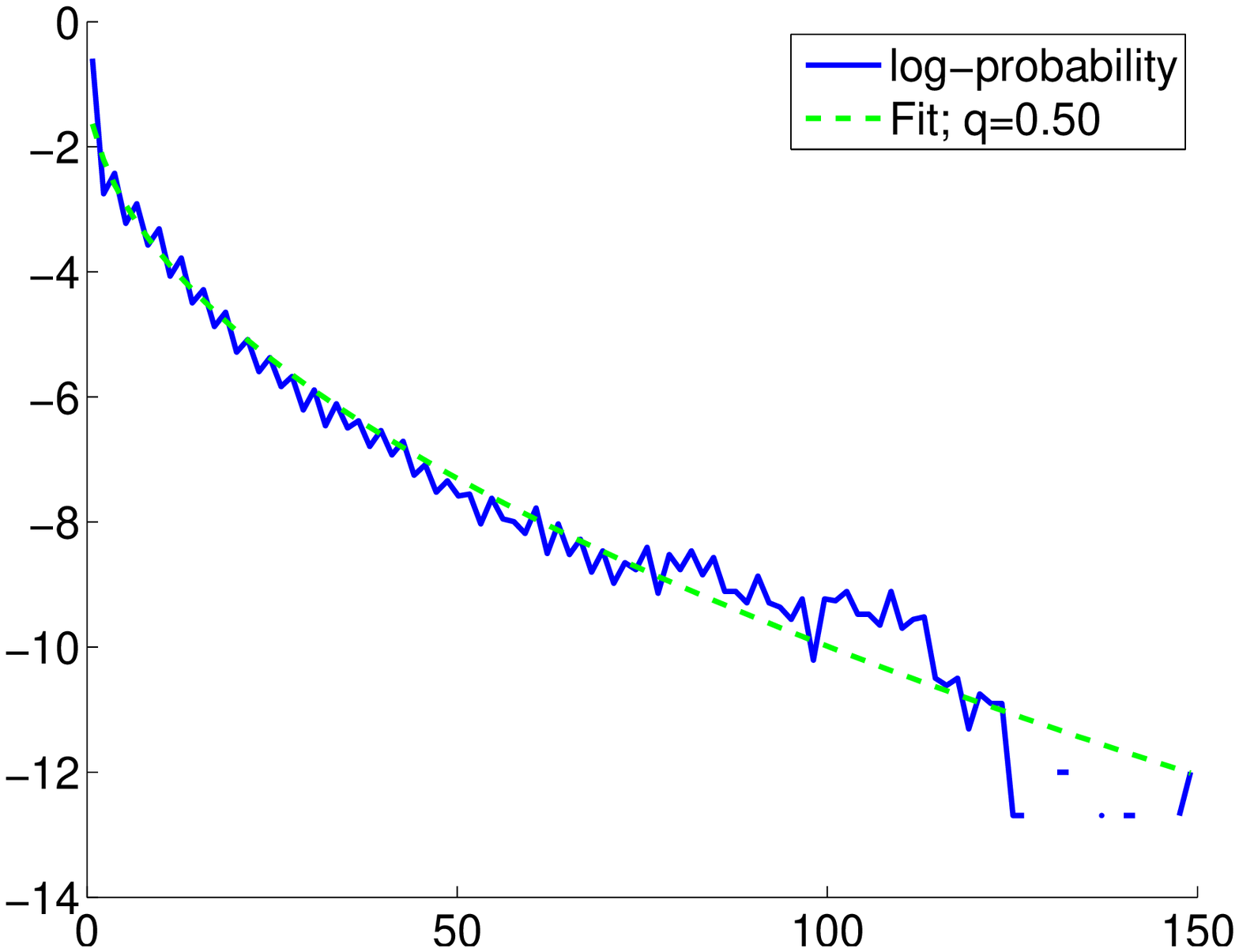}
            \caption{Log-histogram and $t^q$ model}
            \label{fig:dock:fit}
        \end{subfigure}
        \begin{subfigure}[t]{\ws}
            \centering
            \includegraphics[width=\ws,keepaspectratio]{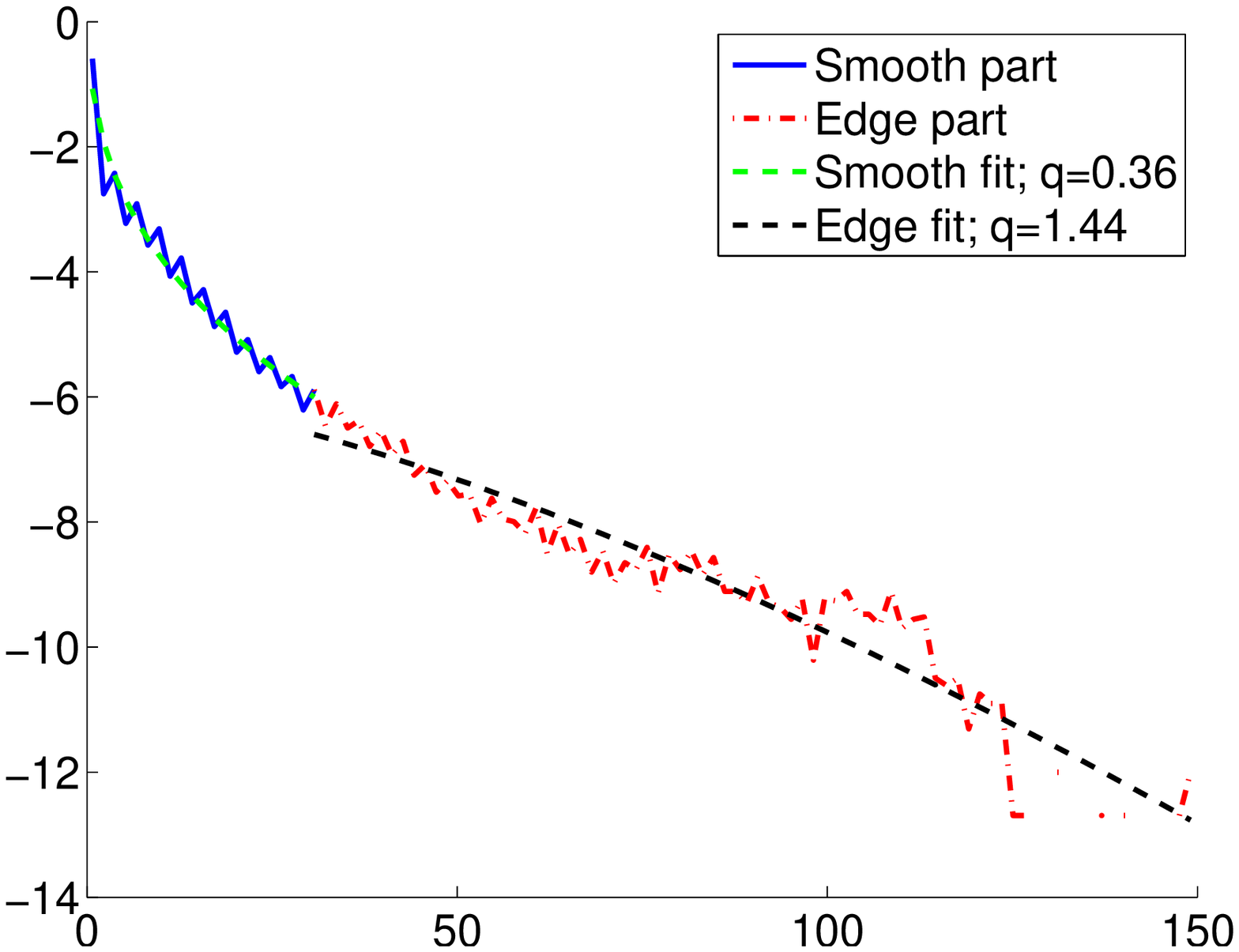}
            \caption{Separate $t^q$ model for edge part}
            \label{fig:dock:fitsplit}
        \end{subfigure}
        \begin{subfigure}[t]{\ws}
            \centering
            \includegraphics[width=\ws,keepaspectratio]{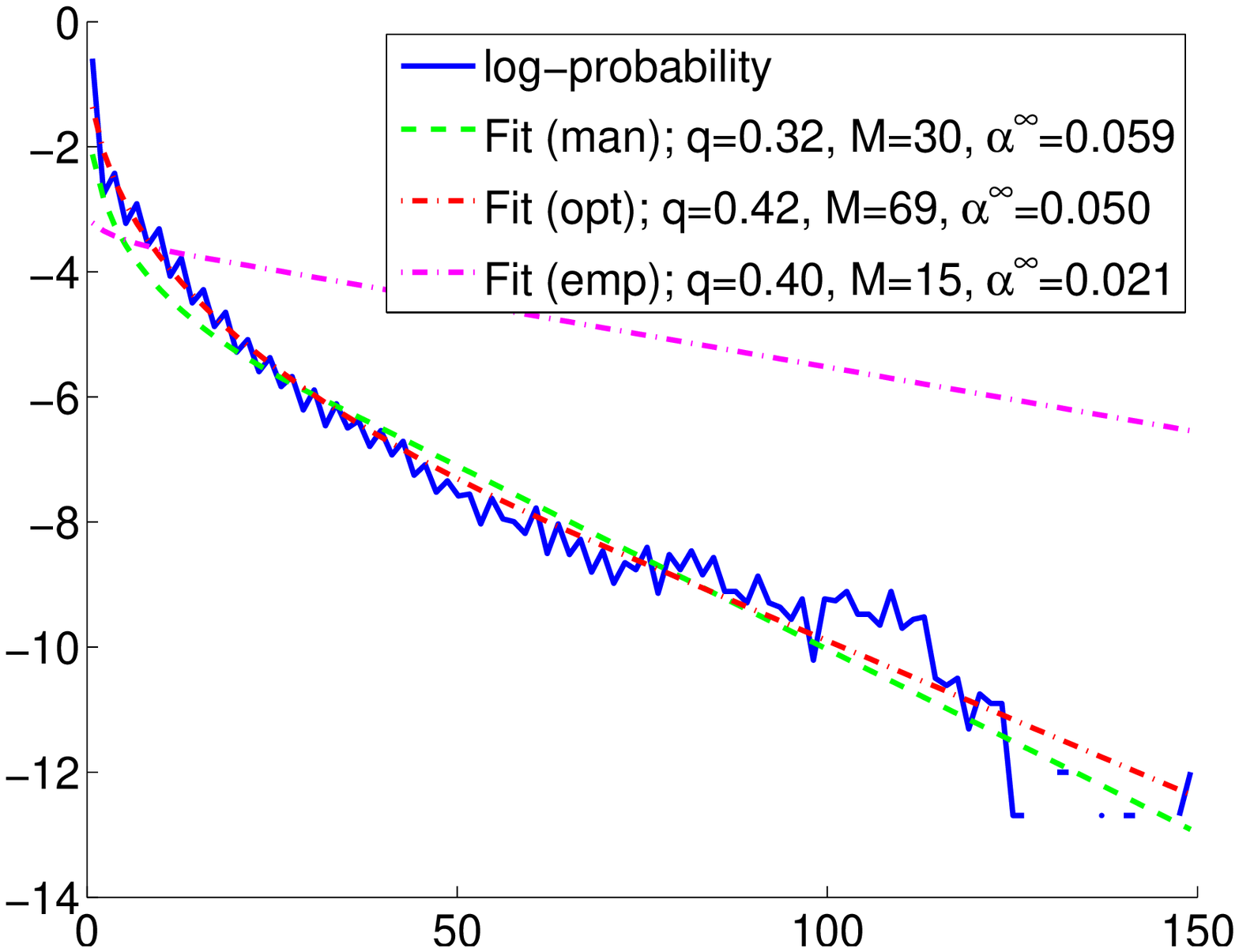}
            \caption{Linearised models}
            \label{fig:dock:fitlin}
        \end{subfigure}
        }
        \caption{Pier photo: discrete gradient histogram and least squares models.
            The image intensity in (\subref{fig:dock:orig}) is in the range $[0,255]$, 
            and we have chosen pixels $k$ with $|\grad u(k)| \ge \IMGthres$
            as edges (\subref{fig:dock:edge}).
            The log-histogram of $|\grad u(k)|$ with
            optimal fit of $t \mapsto t^q$ is displayed in (\subref{fig:dock:fit}).
            This is done separately for the edge pixels
            in (\subref{fig:dock:fitsplit}).
            The linearised model is fitted in (\subref{fig:dock:fitlin})
            for the cut-off point $M=\IMGthres$ (manual edge detection),
            $M=\IMGlinoptthres$ (optimal least squares fit).
            We moreover show the empirically best linearised model.
            }
        \label{fig:dock}
    \end{figure}

    \begin{figure}[!ht]
        \input{img/parrot-data.tex}
        \centering
        \minione{
        \begin{subfigure}[t]{\w}
            \centering
            \includegraphics[width=\w,keepaspectratio]{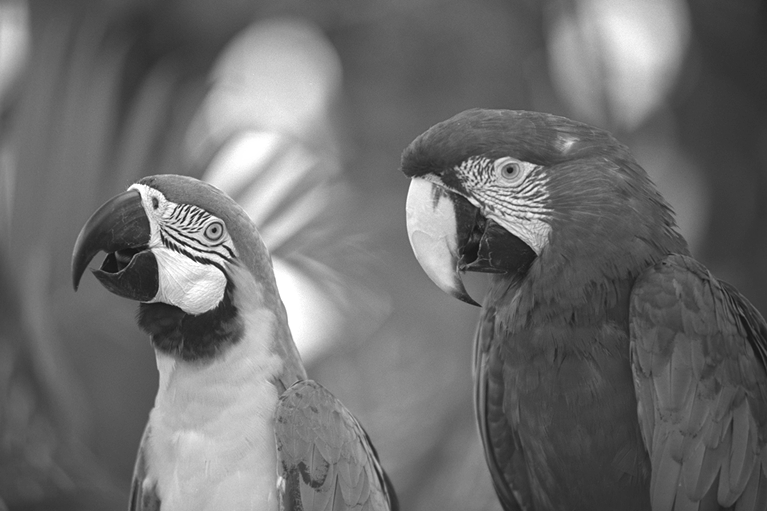}
            \caption{Parrot photo}
            \label{fig:parrot:orig}
        \end{subfigure}
        \begin{subfigure}[t]{\w}
            \centering
            \includegraphics[width=\w,keepaspectratio]{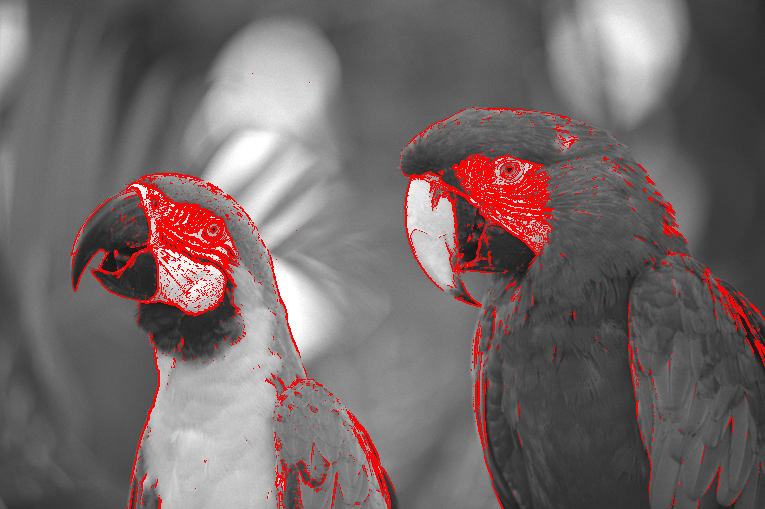}
            \caption{Detected edge pixels (red)}
            \label{fig:parrot:edge}
        \end{subfigure}
        }
        \minitwo{
        \begin{subfigure}[t]{\ws}
            \centering
            \includegraphics[width=\ws,keepaspectratio]{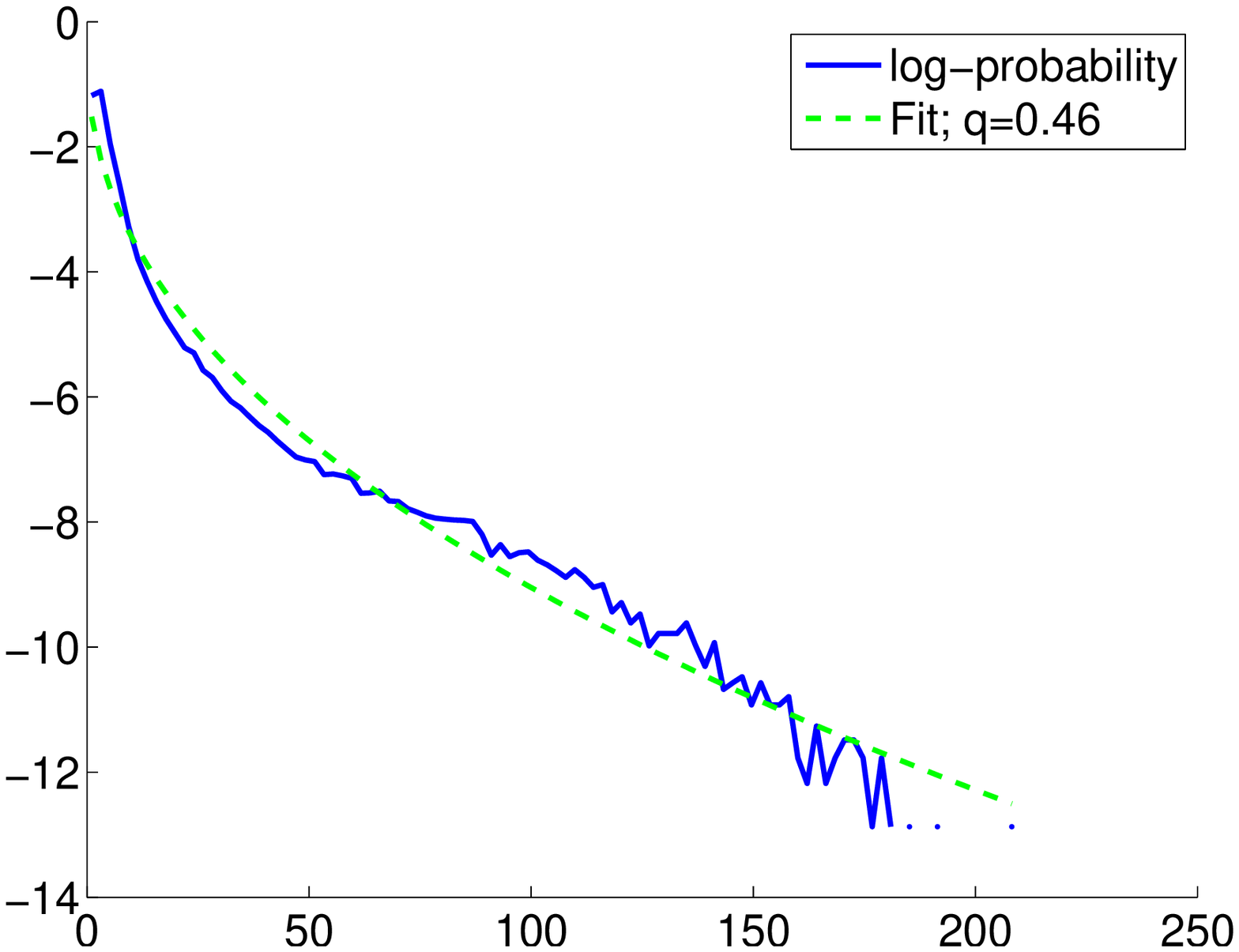}
            \caption{Log-histogram and $t^q$ model}
            \label{fig:parrot:fit}
        \end{subfigure}
        \begin{subfigure}[t]{\ws}
            \centering
            \includegraphics[width=\ws,keepaspectratio]{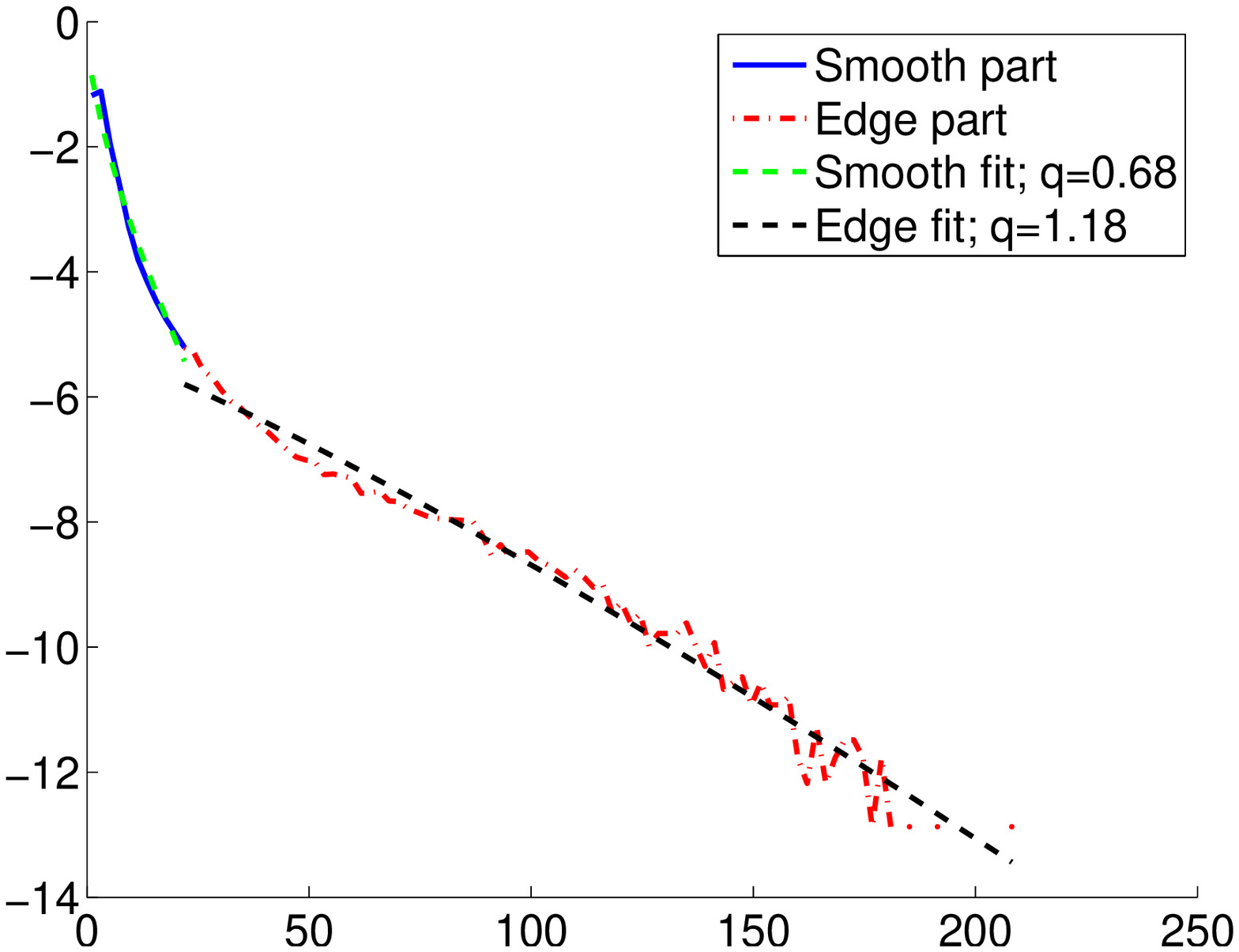}
            \caption{Separate $t^q$ model for edge part}
            \label{fig:parrot:fitsplit}
        \end{subfigure}
        \begin{subfigure}[t]{\ws}
            \centering
            \includegraphics[width=\ws,keepaspectratio]{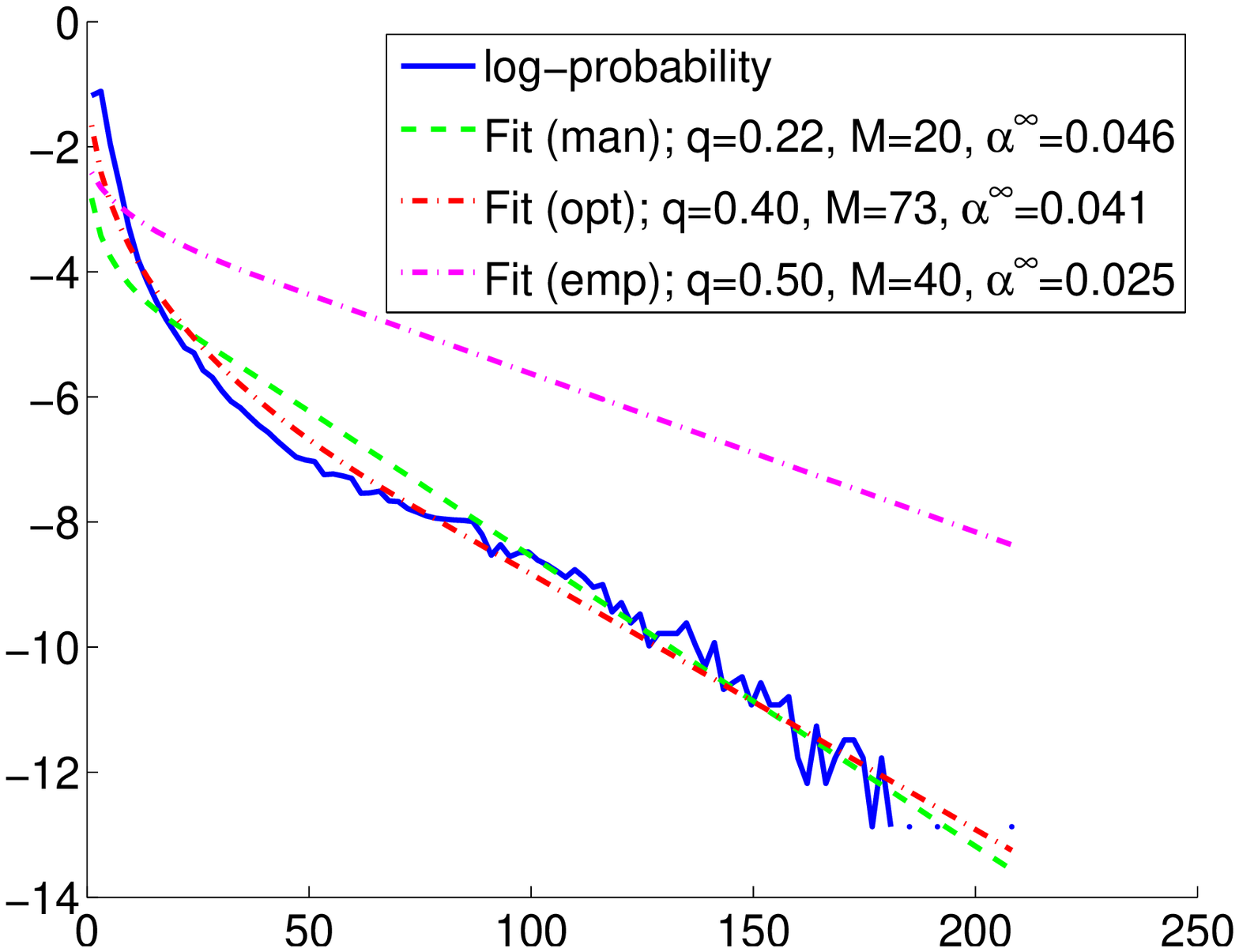}
            \caption{Linearised models}
            \label{fig:parrot:fitlin}
        \end{subfigure}
        }
        \caption{Parrot photo: discrete gradient histogram and least squares models.
            The image intensity in (\subref{fig:parrot:orig}) is in the range $[0,255]$, 
            and we have chosen pixels $k$ with $|\grad u(k)| \ge \IMGthres$
            as edges (\subref{fig:parrot:edge}).
            The log-histogram of $|\grad u(k)|$ with
            optimal fit of $t \mapsto t^q$ is displayed in (\subref{fig:parrot:fit}).
            This is done separately for the edge pixels
            in (\subref{fig:parrot:fitsplit}).
            The linearised model is fitted in (\subref{fig:parrot:fitlin})
            for the cut-off point $M=\IMGthres$ (manual edge detection),
            $M=\IMGlinoptthres$ (optimal least squares fit).
            We moreover show the empirically best linearised model.
            }
        \label{fig:parrot}
    \end{figure}

\setlength{\w}{0.33\textwidth}
\setlength{\ws}{0.41\textwidth}

\def\minione#1{\parbox{\w}{#1}}
\def\minitwo#1{\parbox{\ws}{#1}}

    \begin{figure}[!ht]
        \input{img/summer-data.tex}
        \centering
        \minione{
        \begin{subfigure}[t]{\w}
            \centering
            \includegraphics[width=\w,keepaspectratio]{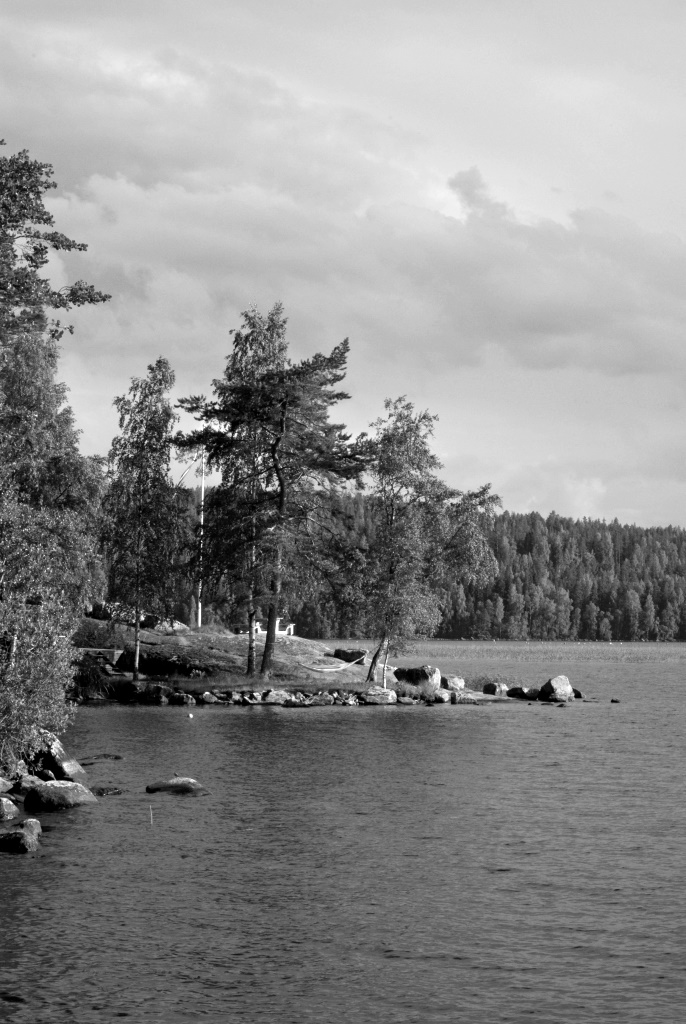}
            \caption{Summer photo}
            \label{fig:summer:orig}
        \end{subfigure}
        \begin{subfigure}[t]{\w}
            \centering
            \includegraphics[width=\w,keepaspectratio]{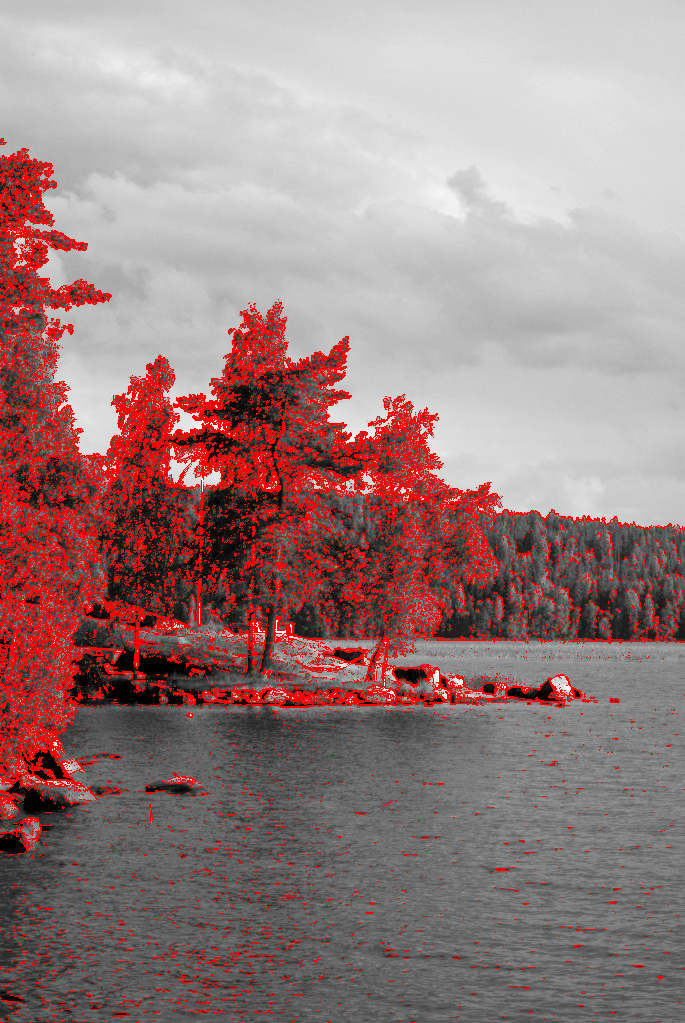}
            \caption{Detected edge pixels (red)}
            \label{fig:summer:edge}
        \end{subfigure}
        }
        \minitwo{
        \begin{subfigure}[t]{\ws}
            \centering
            \includegraphics[width=\ws,keepaspectratio]{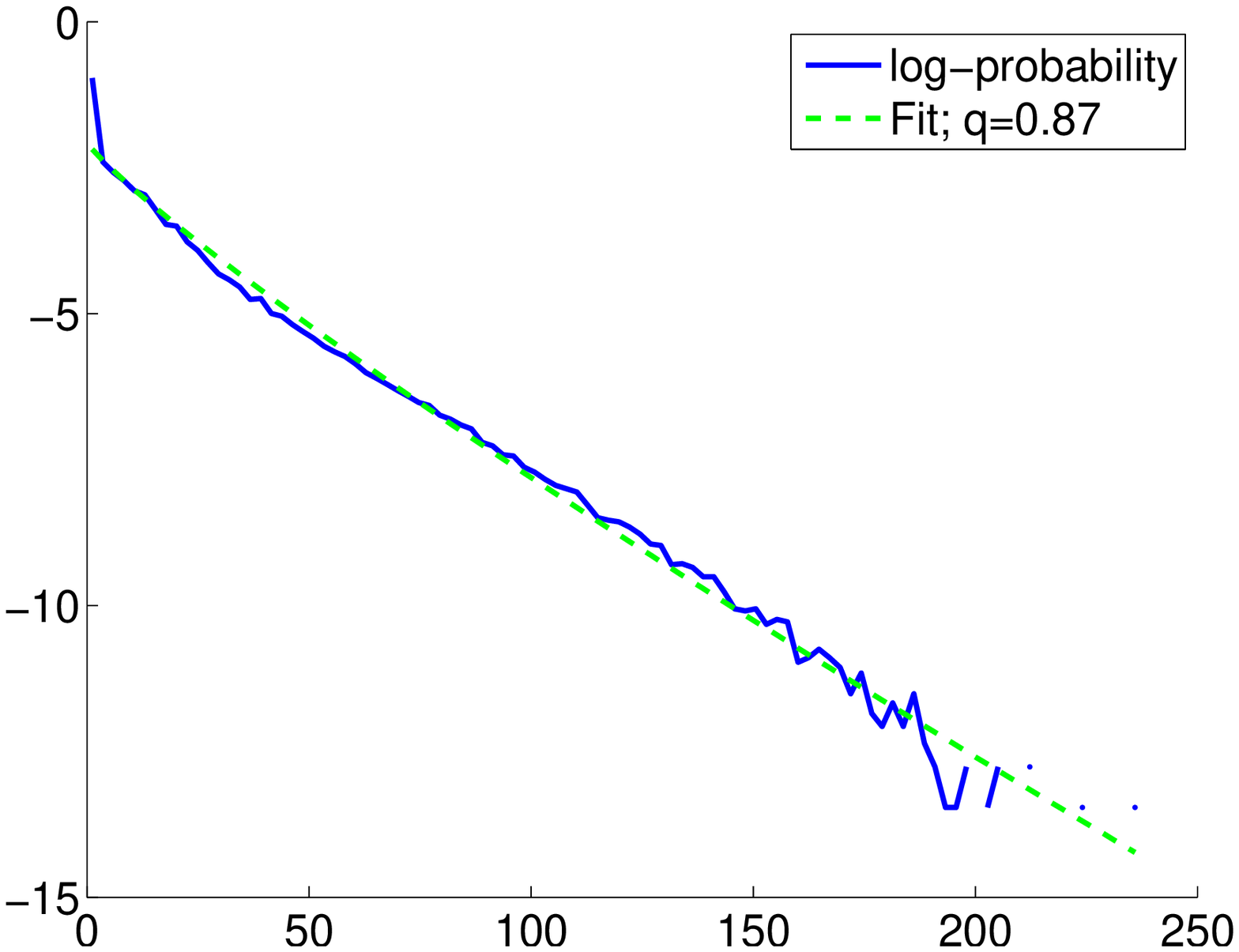}
            \caption{Log-histogram and $t^q$ model}
            \label{fig:summer:fit}
        \end{subfigure}
        \begin{subfigure}[t]{\ws}
            \centering
            \includegraphics[width=\ws,keepaspectratio]{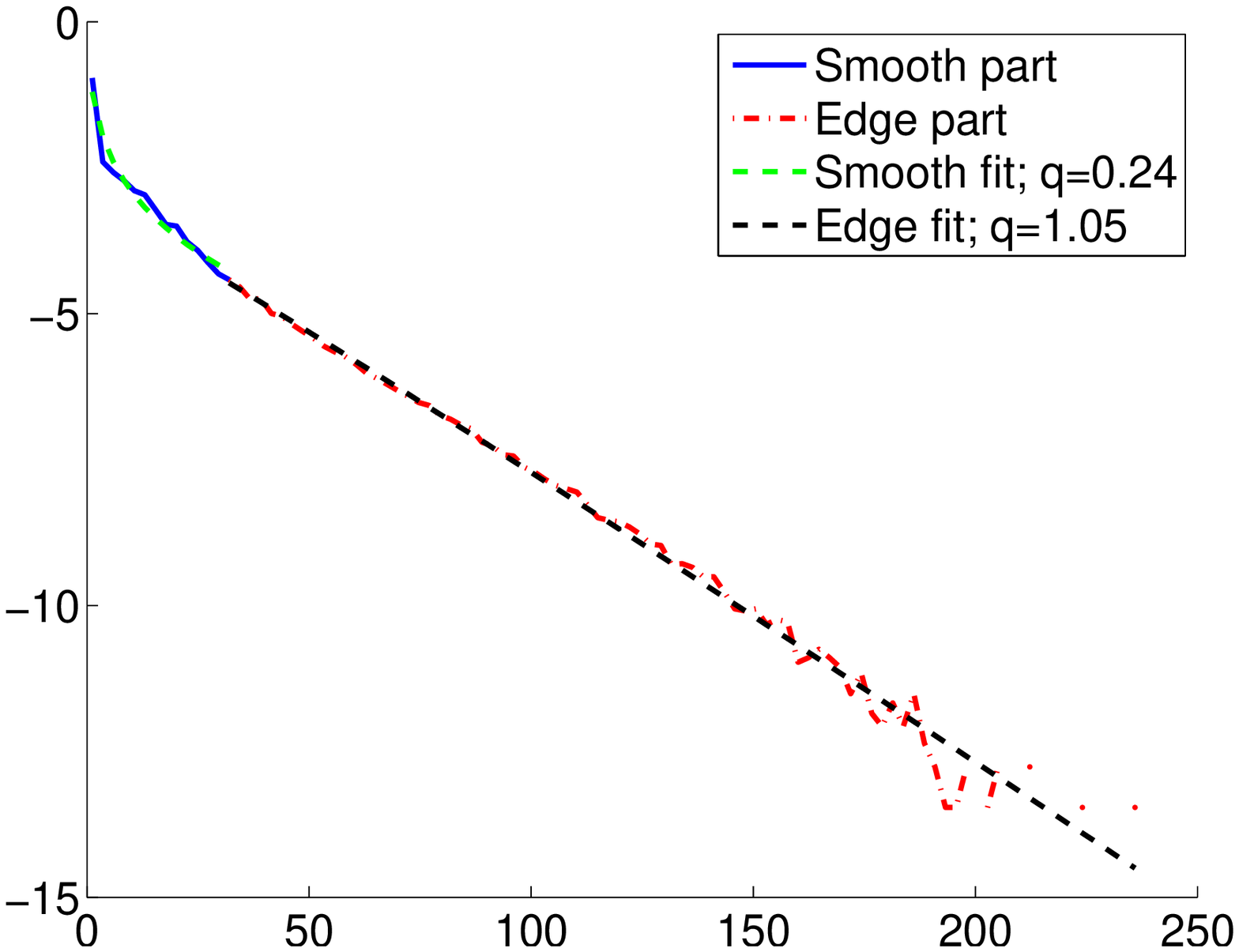}
            \caption{Separate $t^q$ model for edge part}
            \label{fig:summer:fitsplit}
        \end{subfigure}
        \begin{subfigure}[t]{\ws}
            \centering
            \includegraphics[width=\ws,keepaspectratio]{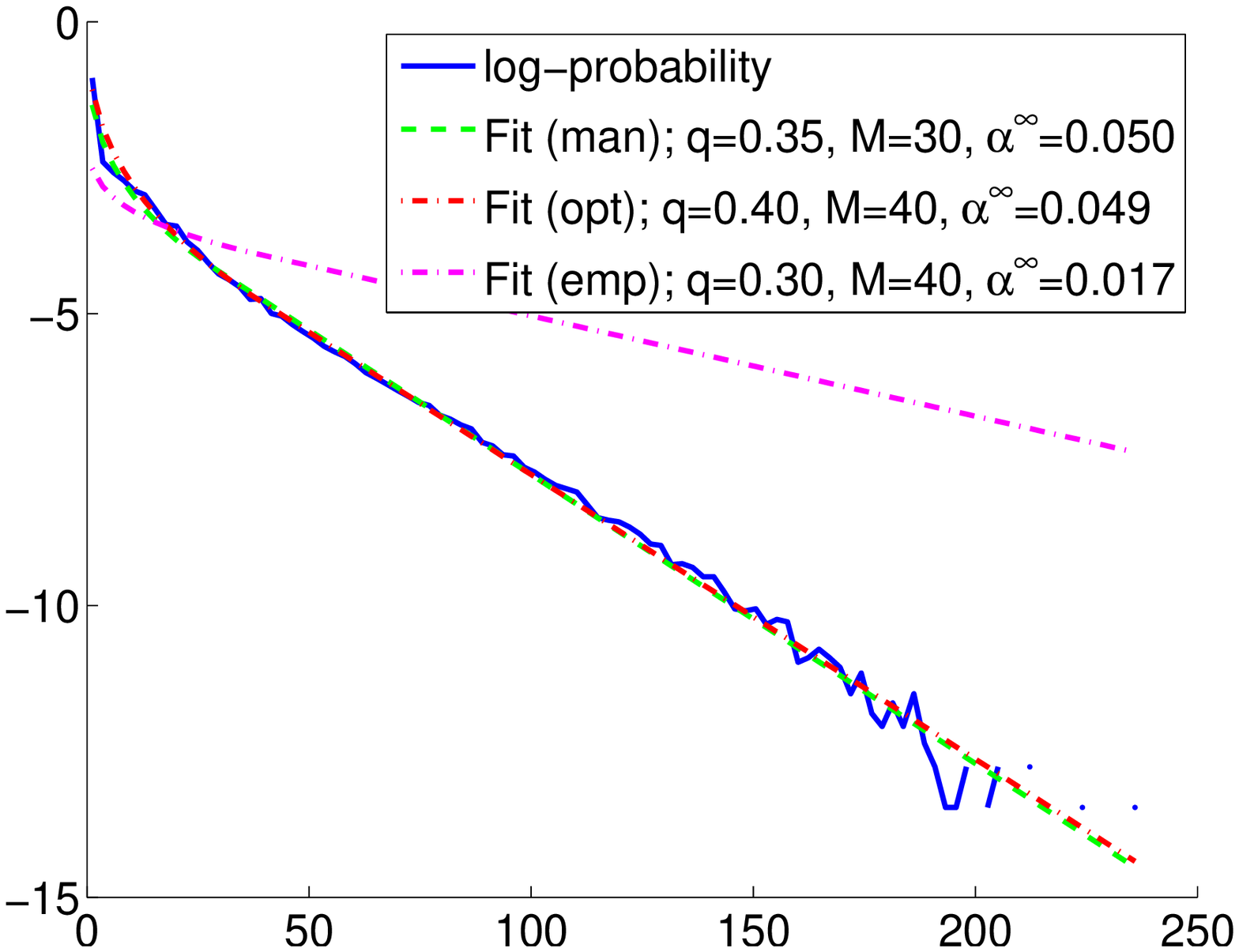}
            \caption{Linearised models}
            \label{fig:summer:fitlin}
        \end{subfigure}
        }
        \caption{Summer photo: discrete gradient histogram and least squares models.
            The image intensity in (\subref{fig:summer:orig}) is in the range $[0,255]$, 
            and we have chosen pixels $k$ with $|\grad u(k)| \ge \IMGthres$
            as edges (\subref{fig:summer:edge}).
            The log-histogram of $|\grad u(k)|$ with
            optimal fit of $t \mapsto t^q$ is displayed in (\subref{fig:summer:fit}).
            This is done separately for the edge pixels
            in (\subref{fig:summer:fitsplit}).
            The linearised model is fitted in (\subref{fig:summer:fitlin})
            for the cut-off point $M=\IMGthres$ (manual edge detection),
            $M=\IMGlinoptthres$ (optimal least squares fit).
            We moreover show the empirically best linearised model.
            }
        \label{fig:summer}
    \end{figure}

Our experiments confirm the findings of
\cite{huang1999statistics} that some $q \in (0, 1)$ is generally
a good fit for the entire distribution, as well as for 
the smooth part. However, optimal $q$ for the edge part varies.
In Figure \ref{fig:dock} actually $q=1.44$ -- larger than one! 
We have to admit that
the number of edge pixels in this image is quite small, so 
statistically the result may be considered unreliable. In Figure
\ref{fig:summer} with a significant proportion of edge pixels,
we still have $q=1.05$. 
These findings also suggest that
\emph{on average} fitting a single $q \in (0, 1)$ to the entire
statistic (not split into edge and smooth parts) may be right,
but there is significant variation between images in the shape
of the distribution for the edge part. The smooth part generally 
looks roughly similar among our test images.

In order to suggest an improved model for image gradient statistics,
in each of Figures \ref{fig:dock}(e)--\ref{fig:summer}(e), we also fit to the 
statistics of the linearised distribution $P_t(t) = C \exp(-\alpha \phi(t))$ for
\begin{equation}
    \label{eq:phi-linearised}
    \phi(t) \defeq \begin{cases}
        t^q & 0 \le t \le M, \\
        (1-q)M^q+q M^{q-1}t, & t > M.
    \end{cases}
\end{equation}
This is again done by a least squares fit on the logarithm of the distribution.
For the `Fit (man)' curve, we fix the cut-off point $M$ to a hand-picked (manual) edge 
threshold and optimise $(q, \alpha)$. We also optimise over all of the 
parameters $(q, \alpha, M)$. This is the `Fit (opt)' curve.
We note that the \term{asymptotic $\alpha$}, which we define as
\[
    \alpha^\infty \defeq \lim_{t \to \infty} \alpha \phi(t)/t
    =q M^{q-1},
\]
is roughly the same for both of the choices, and generally the curves are close
to each other. As $\alpha^\infty$ describes the behaviour of the model on
edges, and for total variation denoising $\alpha^\infty=\alpha$, we
find $\alpha^\infty$ to be a parameter that should indeed stay roughly constant
between models with different $q$ and $M$. It, however, turns out that
$\alpha^\infty$ as obtained by the simple least squares histogram fit
is in practise bad; it gives far too high regularisation, i.e., a too 
narrow distribution. The problem is that
the simple least squares fit on the logarithm over-emphasises the tail of the
distribution, on which we moreover have very little statistics due to the
discrete nature of the data. This yields far too high
$\alpha^\infty$, i.e., the slope of the linear part is too steep in the figures.
Developing a reliable way to obtain the model from the data is outside the scope 
of the present paper, although it is definitely an interesting subject for future studies.
This is why we have also included the curve `Fit (emp)', which is based on an 
empirical choice of $(\alpha^\infty, M, q)$ from our numerical experiments
in the following Section \ref{sec:numerical}. There we keep $\alpha^\infty$
fixed as we vary $M$ and $q$. We will also incorporate the noise
level $\sigma^2$ into $\alpha$. It turns out that for the empirically
good distribution, $q$ is close to the values found by histogram
fitting above, but $\alpha^\infty$ is very different.

\section{Numerical reconstructions}
\label{sec:numerical}

\input{numerics}

\newcommand{\resimg}[7]{
    \begin{figure}
        \input{img/#1-data.tex}%
        \centering%
        \setlength{\w}{#4}%
        \def\resimgnoisy##1##2##3{%
                \begin{subfigure}[t]{\w}%
                    \centering%
                    \includegraphics[width=\w,keepaspectratio]{results/##1-thr=##2}%
                    \caption{$M=##3$}%
                    \label{fig:res-##1:thr-##2}%
                \end{subfigure}%
        }%
        \begin{subfigure}[t]{\w}%
            \centering%
            \includegraphics[width=\w,keepaspectratio]{img/#1-orig}%
            \caption{Original}%
            \label{fig:res-#1:orig}%
        \end{subfigure}
        \begin{subfigure}[t]{\w}%
            \centering%
            \includegraphics[width=\w,keepaspectratio]{img/#1-noisy}%
            \caption{Noisy image}%
            \label{fig:res-#1:noisy}%
        \end{subfigure}
        \def\resimgnoisythis##1##2{\resimgnoisy{#1}{##1}{##2}}%
        #7%
        \caption{#2: denoising results with noise level $\sigma=#3$ (Gaussian), for varying cut-off $M$,
            fixed $q=#5$ and fixed $\alpha^\infty=#6$.}
        \label{fig:res-#1}
    \end{figure}
}

\def\ssimoptimal#1{#1 \text{ (SSIM-optimal)}}
\def\psnroptimal#1{#1 \text{ (PSNR-optimal)}}

    \begin{figure}
        \input{img/dock-data.tex}%
        \centering%
        \setlength{\w}{0.495\textwidth}%
        \def\resimgnoisy#dock#Pier photo#30{%
                \begin{subfigure}[t]{\w}%
                    \centering%
                    \includegraphics[width=\w,keepaspectratio]{results/#dock-thr=#Pier photo}%
                    \caption{$M=#30$}%
                    \label{fig:res-#dock:thr-#Pier photo}%
                \end{subfigure}%
        }%
        \begin{subfigure}[t]{\w}%
            \centering%
            \includegraphics[width=\w,keepaspectratio]{img/dock-orig}%
            \caption{Original}%
            \label{fig:res-dock:orig}%
        \end{subfigure}
        \begin{subfigure}[t]{\w}%
            \centering%
            \includegraphics[width=\w,keepaspectratio]{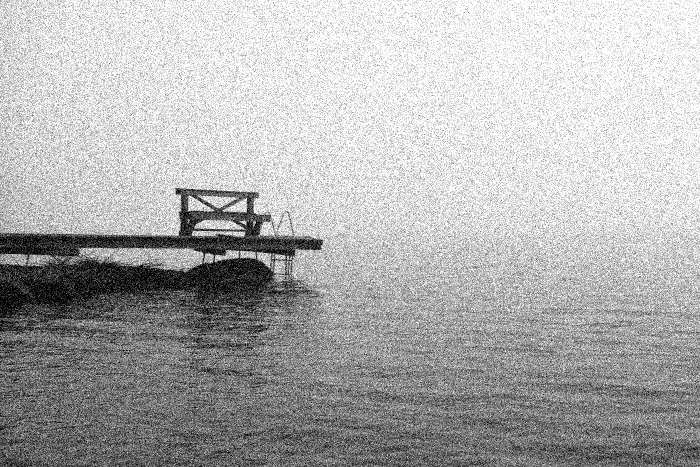}%
            \caption{Noisy image}%
            \label{fig:res-dock:noisy}%
        \end{subfigure}
        \def\resimgnoisythis#dock#Pier photo{\resimgnoisy{dock}{#dock}{#Pier photo}}%
    \resimgnoisythis{0}{0} 
    \resimgnoisythis{10}{\psnroptimal{10}} %
    \resimgnoisythis{40}{\ssimoptimal{40}} %
    \resimgnoisythis{Inf}{\infty}%
        \caption{Pier photo denoising results with noise level $\sigma=30$ (Gaussian), for varying cut-off $M$,
            fixed $q=0.4$ and fixed $\alpha^\infty=0.0207
    $.}
        \label{fig:res-dock}
    \end{figure}

    \begin{figure}
        \input{img/parrot-data.tex}%
        \centering%
        \setlength{\w}{0.495\textwidth}%
        \def\resimgnoisy#parrot#Parrot photo#30{%
                \begin{subfigure}[t]{\w}%
                    \centering%
                    \includegraphics[width=\w,keepaspectratio]{results/#parrot-thr=#Parrot photo}%
                    \caption{$M=#30$}%
                    \label{fig:res-#parrot:thr-#Parrot photo}%
                \end{subfigure}%
        }%
        \begin{subfigure}[t]{\w}%
            \centering%
            \includegraphics[width=\w,keepaspectratio]{img/parrot-orig}%
            \caption{Original}%
            \label{fig:res-parrot:orig}%
        \end{subfigure}
        \begin{subfigure}[t]{\w}%
            \centering%
            \includegraphics[width=\w,keepaspectratio]{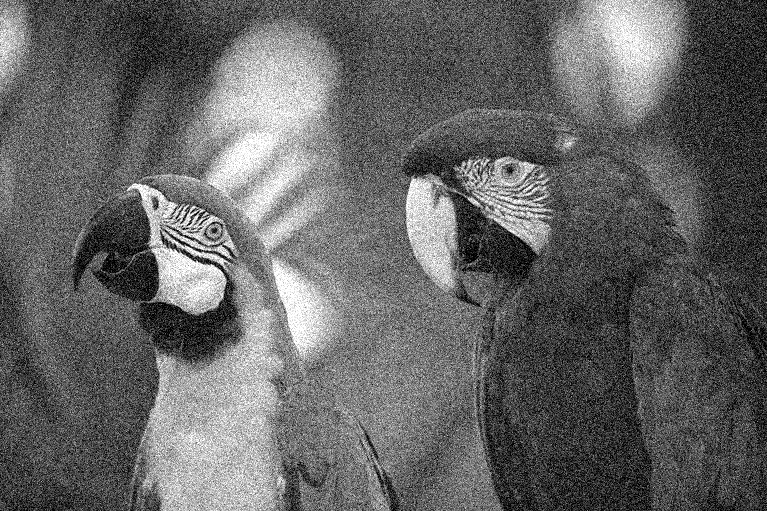}%
            \caption{Noisy image}%
            \label{fig:res-parrot:noisy}%
        \end{subfigure}
        \def\resimgnoisythis#parrot#Parrot photo{\resimgnoisy{parrot}{#parrot}{#Parrot photo}}%
    \resimgnoisythis{0}{\psnroptimal{0}} 
    \resimgnoisythis{15}{\ssimoptimal{15}} %
    \resimgnoisythis{40}{40}
    \resimgnoisythis{Inf}{\infty}%
        \caption{Parrot photo denoising results with noise level $\sigma=30$ (Gaussian), for varying cut-off $M$,
            fixed $q=0.5$ and fixed $\alpha^\infty=0.0253
    $.}
        \label{fig:res-parrot}
    \end{figure}

    \begin{figure}
        \input{img/summer-data.tex}%
        \centering%
        \setlength{\w}{0.325\textwidth}%
        \def\resimgnoisy#summer#Summer photo0.00430
    0{%
                \begin{subfigure}[t]{\w}%
                    \centering%
                    \includegraphics[width=\w,keepaspectratio]{results/#summer-thr=#Summer photo}%
                    \caption{$M=0.00430
    0$}%
                    \label{fig:res-#summer:thr-#Summer photo}%
                \end{subfigure}%
        }%
        \begin{subfigure}[t]{\w}%
            \centering%
            \includegraphics[width=\w,keepaspectratio]{img/summer-orig}%
            \caption{Original}%
            \label{fig:res-summer:orig}%
        \end{subfigure}
        \begin{subfigure}[t]{\w}%
            \centering%
            \includegraphics[width=\w,keepaspectratio]{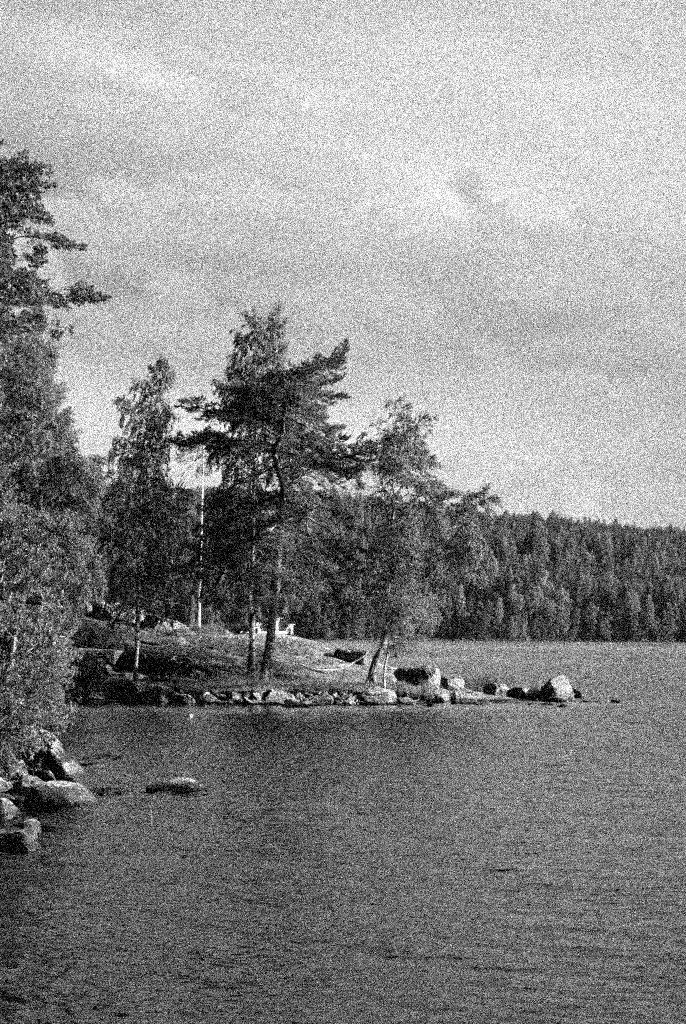}%
            \caption{Noisy image}%
            \label{fig:res-summer:noisy}%
        \end{subfigure}
        \def\resimgnoisythis#summer#Summer photo{\resimgnoisy{summer}{#summer}{#Summer photo}}%
    \resimgnoisythis{0}{0} 
    \resimgnoisythis{20}{\psnroptimal{20}} %
    \resimgnoisythis{40}{\ssimoptimal{40}} %
    \resimgnoisythis{Inf}{\infty}%
        \caption{Summer photo denoising results with noise level $\sigma=60$ (Gaussian), for varying cut-off $M$,
            fixed $q=0.3$ and fixed $\alpha^\infty=0.00430
    $.}
        \label{fig:res-summer}
    \end{figure}

\begin{table}[!ht]
    \caption{Denoising results of all the three test photos.
        The noise level $\sigma$, exponent $q$ and asymptotic
        regularisation $\alpha^\infty$ are fixed. 
        The cut-off point $M$ varies. We report both the PSNR and the SSIM,
        with the optimal values underlined.}
    \label{table:res}
    \centering
    \begin{tabular}{lllllllll}%
        \hline
        \multicolumn{9}{l}{\textbf{Parrot photo} / $\sigma=30$, $q=0.5$, $\alpha^\infty=0.0253$}
        \\\hline
        $M$&0&10&15 
        &30&40&50&60&$\infty$\\
        PSNR&\underline{30.3432}&29.9156&29.5876 
        &28.8098&28.3655&28.0398&27.7424&28.6829\\
        SSIM&0.7914&0.7919&\underline{0.7922} 
        &0.7906&0.7887&0.7854&0.7823&0.7552\\
        \hline
        \multicolumn{9}{l}{\textbf{Pier photo} / $\sigma=30$, $q=0.4$, $\alpha^\infty=0.0207$}
        \\\hline
        $M$&0&10&20&30&40&50&60&$\infty$\\
        PSNR&29.0019&\underline{29.3959}&29.3472&29.0918&28.8988&28.679&28.5327&28.5836\\
        SSIM&0.6737&0.7191&0.7477&0.7556&\underline{0.7619}&0.7608&0.7593&0.7297\\
        \hline
        \multicolumn{9}{l}{\textbf{Summer photo} / $\sigma=30$, $q=0.4$, $\alpha^\infty=0.00430$}
        \\\hline
        $M$&0&10&20&30&40&50&60&$\infty$\\
        PSNR&26.0919&26.2750&\underline{26.2851}&26.0643&25.7443&25.3627&25.0449&25.2755\\
        SSIM&0.589&0.6175&0.6489&0.6641&\underline{0.6644}&0.6615&0.6543&0.6175\\
    \end{tabular}
\end{table}

We report in Figures \ref{fig:res-dock}--\ref{fig:res-summer} and 
Table \ref{table:res} the results of denoising our three test images
using this algorithm with rather high artificial noise levels. 
We have added Gaussian noise of standard deviation
$\sigma=30$ to all test images.
We report both the conventional peak-signal-to-noise ratio (PSNR) as well as
the structural similarity measure (SSIM) of \cite{wang2004ssim}. The
latter better quantifies the visual quality of images by
essentially computing the PSNR in local windows and combining the results
in a non-linear fashion. The range of the SSIM is $[0, 1]$, the higher
the better.

In our computations, we keep $\alpha^\infty$ and $q$ fixed, and vary $M$ (by altering $\alpha$ as
necessary). For $M=\infty$, i.e., the original $\TV^q$ model, we simply take $\alpha$ as our chosen fixed
$\alpha^\infty$. This is because $\phi^\infty=0$, so the real asymptotic alpha for the model is always zero.
It is quite remarkable that in our results fine features of the images are always retained
very well although higher $M$ tends to increase the stair-casing effect (not applicable to $M=\infty$).
At the optimal choice of $M$ by PSNR or SSIM, more noise can be seen to be 
removed than by $\TV$ ($M=0$). Generally, we can say that adding the cut-off
$M$ improves the results compared to the earlier $\TV^q$ model without 
cut-off ($M=\infty$). Whether the results are better than conventional $\TV$ 
denoising is open to debate. By PSNR and SSIM the results tend to favor the TV$^\varphi$-model.
Visually oscillatory effects of noise are better removed, but at
the same time the stair-casing is accentuated. The best result is
in the eye of the beholder.

We also tested on the parrot photo the effect of the multiscale regularisation
term $\eta$ by including the first term $\eta_1$ of the sum for varying 
weights of $\eta_0$ and convolution kernel widths $\epsilon_1$.
The results are in Figure \ref{fig:reseta-parrot} 
and Table \ref{table:reseta}. Clearly large $\eta_0$ has a deteriorating
effect on both PSNR and SSIM, whereas the effect of the choice of
$\epsilon_1$ is less severe. Visually, large $\eta_0$ creates an
almost artistic quantisation and feature-filtering effect. The
latter is also controlled by $\epsilon_1$: large $\epsilon_1$
tends to remove large features. A particular feature to notice
is the eye of the parrot on the right in
Figure \ref{fig:reseta-parrot}(\subref{fig:reseta-parrot:beta=7d071e-03_eps=1d0e+00})
versus (\subref{fig:reseta-parrot:beta=7d071e-03_eps=2d0e+00}).
It has disappeared altogether in the latter.

\newcommand{\resimgeta}[8]{
    \begin{figure}[!ht]
        \input{img/#1-data.tex}%
        \centering%
        \setlength{\w}{#4}%
        \def\resimgnoisy##1##2##3{%
                \begin{subfigure}[t]{\w}%
                    \centering%
                    \includegraphics[width=\w,keepaspectratio]{results/eta_test_##1_##2}%
                    \caption{##3}%
                    \label{fig:reseta-##1:##2}%
                \end{subfigure}%
        }
        \def\resimgnoisythis##1##2{\resimgnoisy{#1}{##1}{##2}}%
        #8%
        \caption{Effect of the $\eta$ term on the #2. 
            Only the first term $\eta_1$ of the sum is included,
            with varying convolution width $\epsilon_1$
            and weight $\eta_0$.
            The noise level $\sigma=#3$ (Gaussian), cut-off
            $M=#7$, exponent $q=#5$ and asymptotic regularisation
            $\alpha^\infty=#6$ are fixed.}
        \label{fig:reseta-#1}
    \end{figure}
}

    \begin{figure}[!ht]
        \input{img/parrot-data.tex}%
        \centering%
        \setlength{\w}{0.495\textwidth}%
        \def\resimgnoisy#parrot#parrot photo#30{%
                \begin{subfigure}[t]{\w}%
                    \centering%
                    \includegraphics[width=\w,keepaspectratio]{results/eta_test_#parrot_#parrot photo}%
                    \caption{#30}%
                    \label{fig:reseta-#parrot:#parrot photo}%
                \end{subfigure}%
        }
        \def\resimgnoisythis#parrot#parrot photo{\resimgnoisy{parrot}{#parrot}{#parrot photo}}%
    \resimgnoisythis{beta=7d071e-03_eps=1d0e+00}{$\eta_0=7.071\ee^{-03}$, $\epsilon_1=1.0$}
    \resimgnoisythis{beta=7d071e-03_eps=2d0e+00}{$\eta_0=7.071\ee^{-03}$, $\epsilon_1=2.0$}
    \\
    \resimgnoisythis{beta=7d071e-04_eps=1d0e+00}{$\eta_0=7.071\ee^{-04}$, $\epsilon_1=1.0$}
    \resimgnoisythis{beta=7d071e-04_eps=2d0e+00}{$\eta_0=7.071\ee^{-04}$, $\epsilon_1=2.0$}
    \\
    \resimgnoisythis{beta=7d071e-05_eps=1d0e+00}{$\eta_0=7.071\ee^{-05}$, $\epsilon_1=1.0$}
    \resimgnoisythis{beta=7d071e-05_eps=2d0e+00}{$\eta_0=7.071\ee^{-05}$, $\epsilon_1=2.0$}
        \caption{Effect of the $\eta$ term on the parrot photo. 
            Only the first term $\eta_1$ of the sum is included,
            with varying convolution width $\epsilon_1$
            and weight $\eta_0$.
            The noise level $\sigma=30$ (Gaussian), cut-off
            $M=10$, exponent $q=0.5$ and asymptotic regularisation
            $\alpha^\infty=0.0253
    $ are fixed.}
        \label{fig:reseta-parrot}
    \end{figure}

\begin{table}[!ht]
    \caption{Effect of the $\eta$ term on the parrot photo. 
        Only the first term $\eta_1$ of the sum is included,
        with varying convolution width $\epsilon_1$
        and weight $\eta_0$.
        The noise level $\sigma=30$ (Gaussian), cut-off
        $M=10$, exponent $q=0.5$ and asymptotic regularisation
        $\alpha^\infty=0.0253$ are fixed.
        Optimal SSIM and PSNR are underlined.}
    \label{table:reseta}
    \centering
    \begin{tabular}{l|ll|ll|ll}%
        &\multicolumn{2}{c|}{$\epsilon_1=0.5$}&
        \multicolumn{2}{c|}{$\epsilon_1=1$}&
        \multicolumn{2}{c}{$\epsilon_1=2$}
        \\\hline
        &PSNR&SSIM&PSNR&SSIM&PSNR&SSIM
        \\\hline
        $\eta_0=0.1\alpha$ &\underline{29.9433}&0.8036&29.8023&\underline{0.8091}&29.6197&0.8085 \\
        $\eta_0=\alpha$ &29.2700&0.8072&28.2237&0.8014&27.5659&0.7970 \\
        $\eta_0=10\alpha$ &26.8069&0.7939&25.7874&0.7921&24.9003&0.7886 \\
    \end{tabular}
\end{table}

\section{Conclusion}
\label{sec:concl}

We have studied difficulties with non-convex total variation models in
the function space setting. We have demonstrated that the model
\eqref{eq:tvqd} continues to do what it is proposed to do in the
discrete setting -- to promote piecewise constant solutions -- for most, 
but not all, energies $\phi$, employed in the literature. Naïve forms of
the model \eqref{eq:tvqc}, proposed to model real gradient 
distributions in images, however have much more severe difficulties.
We have shown that the model can be remedied if we replace the topology
of weak* convergence by that of area strict convergence. In order to
do this, we have to add additional multiscale regularisation in terms 
of the functional $\eta$ into the model, and to linearise the energy
$\phi$ for large gradients. The latter is needed to make the model 
BV-coercive, and to have any kind of penalisation for jumps. We have
demonstrated through numerical experiments and simple statistics that
this model, in fact, better matches reality than the simple energies
$\phi(t)=t^q$. \emph{Our purely theoretical starting point has therefore 
led to improved practical models.} The $\eta$ functional, however, 
remains a ``theoretical artefact''. It has its own regularisation 
effect that, naturally, does not distort the results too much for 
small parameters (though it does so for large parameters). 
As we shown in Proposition \ref{prop:partial-eta-sum-limit}, it 
can be ignored in discretisations when not passing to the function
space limit.

\section*{Acknowledgements}

While in Cambridge, T.~Valkonen has been financially supported by the 
King Abdullah University of Science and Technology (KAUST) Award 
No.~KUK-I1-007-43, and the EPSRC first grant Nr.~EP/J009539/1
``Sparse \& Higher-order Image Restoration''. In Quito, T.~Valkonen
has been supported by a Prometeo scholarship of the Senescyt.
M.~Hinterm\"uller and T.~Wu have been supported by the Austrian
FWF SFB F32 ``Mathematical Optimization in Biomedical Sciences'', 
and the START-Award Y305.

\bibliography{abbrevs,bib,bib-own}

\appendix

\section{Vectorial $\eta$ functional}
\label{app:vecteta}

\newcommand{\INDEXX}{n}
\newcommand{\INDEXY}{\ell}

We now study a condition ensuring the convergence
of the total variation $\abs{\mu^i}(\Omega)$ subject to the
weak* convergence of the measures $\mu^i \in \Meas(\Omega; \R^m)$, ($i=0,1,2,\ldots$).
Improving a result first presented in \cite{tuomov-bd,tuomov-ap1},
we show in Theorem \ref{thm:eta} below that
if $\{f_\ell\}_{\ell=0}^\infty$ is a normalised nested
sequence of functions as in Definition \ref{def:nested}
below, then it suffices to bound
\begin{equation}
    \label{eq:def-eta-n}
    \eta(\mu^i) \defeq
        \sum_{\ell=0}^\infty \eta_\ell(\mu^i),
        \quad
        \text{where}
        \quad
        \eta_\ell(\mu^i) \defeq 
        \int \abs{\mu^i}(\tau_x f_\ell) - \abs{\mu^i(\tau_x f_\ell)} \d x.
    \quad (\mu \in \Meas(\Omega; \R^N)).
\end{equation}
Here we employ the notation $\tau_x f(y) \defeq f(y-x)$.
Also, we write $\abs{\mu^i(\tau_x f_\ell)} \defeq \norm{\mu^i(\tau_x f_\ell)}_2$.

\begin{definition}
    \label{def:nested}
    Let $f_\ell: \R^m \to \R$, ($\ell=0,1,2,\ldots$),
    be bounded Borel functions with compact support
    that are continuous in $\R^m \setminus S_{f_\ell}$,
    i.e.~the approximate discontinuity set equals
    the discontinuity set.
    Also let $\{\nu_\ell\}_{\ell=0}^\infty$ be a sequence in $\Meas(\R^m)$ with
    $\abs{\nu_\ell}(\R^m)=1$.
    The sequence $\{(f_\ell, \nu_\ell)\}_{\ell=0}^\infty$ is then
    said to form a \term{nested sequence of functions}
    if $f_\ell(x) = \int f_{\ell+1}(x-y) \d\nu_\ell(y)$ (a.e.).
    The sequence is said to be \term{normalised} if
    $f_\ell \ge 0$ and $\int f_\ell \d x = 1$.
\end{definition}

\begin{example}
    Let $\rho$ be the standard convolution mollifier such that
    \[
        \rho(x)
        \defeq
        \begin{cases}
            \exp(-1/(1-\norm{x}^2)), & \norm{x} < 1, \\
            0, & \norm{x} \ge 1,
        \end{cases}
    \]
    and define
    $\rho_\epsilon(x) \defeq \epsilon^{-m}\rho(x/\epsilon)$.
    Since $\rho_{\epsilon+\delta}=\rho_\epsilon*\rho_\delta$ where $*$ denotes a convolution,
    we deduce that
    $f_\ell \defeq \rho_{2^{-\ell}}$ and $\nu_\ell=\rho_{2^{-\ell-1}}$
    form a normalised nested sequence.
\end{example}

We require the following basic lemma for our vectorial case.

\begin{lemma}
    \label{lemma:integral-norm-inside}
    Let $\nu \in \Meas(\Omega)$ be a positive Radon measure, and $g \in L^1(\Omega; \R^N)$.
    Then
    \[
        \adaptnorm{\int_\Omega g(x) \d\nu(x)}_2 \le \int_\Omega \norm{g(x)}_2 \d\nu(x).
    \]
\end{lemma}
\begin{proof}
    For any $x \in \Omega$, we write $g(x)=\theta(x) v(x)$ with 
    $v(x)=(v_1(x), \ldots, v_N(x))$, $0 \le \theta(x)$, and $\norm{v(x)}=1$.
    Then we define $\lambda \defeq \theta\nu$. Now
    \[
        \begin{split}
        \adaptnorm{\int_\Omega g(x) \d\nu(x)}_2^2
        &
        =
        \sum_{j=1}^N
        \left(\int_\Omega v_j(x) \d\lambda(x)\right)^2
        =
        \sum_{j=1}^N
        \lambda(\Omega)^2
        \left(\frac{1}{\lambda(\Omega)}\int_\Omega v_j(x) \d\lambda(x)\right)^2
        \\
        &
        \le
        \sum_{j=1}^N
        \lambda(\Omega)
        \int_\Omega \abs{v_j(x)}^2 \d\lambda(x)
        =
        \lambda(\Omega)^2.
        \end{split}
    \]
    Here we have used Jensen's inequality.
    From this we conclude
    \[
        \lambda(\Omega)=\int_\Omega \theta(x) \d\nu(x) = \int_\Omega \norm{g(x)}_2 \d\nu(x),
    \]
    proving the claim.
\end{proof}

With the help of the above lemma, in
\cite{tuomov-ap1} Theorem \ref{thm:eta} below was 
proved exactly as the case of scalar-valued measures ($N=1$). 
Our proof here is however slightly different.
We base it on the following more general lemma on partial sums, 
which we also need for the proof of 
Proposition \ref{prop:partial-eta-sum-limit}.

\begin{lemma}
    \label{lemma:eta-ki}
    Let $\Omega \subset \R^m$ be an open and bounded set,
    and $\{(f_\ell, \nu_\ell)\}_{\ell=0}^\infty$ a normalised nested
    sequence of functions. 
    Let $\{K_i\}_{i=1}^\infty \subset \N^+$ satisfy
    $\lim_{i \to\infty} K_i = \infty$.
    Suppose $\{\mu^i\}_{i=0}^\infty \subset \Meas(\Omega; \R^N)$
    weakly* converges to $\mu \in \Meas(\Omega; \R^N)$ with
    \begin{equation}
        \label{eq:eta-ki-bnd1}
        \sup_i \abs{\mu^i}(\Omega) + \sum_{\ell=1}^{K_i} \eta_\ell(\mu^i) < \infty,
    \end{equation}
    and
    \begin{equation}
        \label{eq:eta-ki-bnd2}
        \eta(\mu) < \infty.
    \end{equation}
    Then
    \begin{equation}
        \label{eq:eta-ell-lsc}
        \eta_\ell(\mu) \le \liminf_{i \to \infty} \eta_\ell(\mu^i),
        \quad
        (\ell=0,1,2,\ldots).
    \end{equation}
    If also $\abs{\mu^i} \weaktostar \lambda$ in $\Meas(\Omega)$,
    then $\lambda=\abs{\mu}$.
    Moreover, provided the weak* convergences hold in 
    $\Meas(\R^m; \R^N)$, resp., $\Meas(\R^m)$, then
    \begin{equation}
        \label{eq:eta-ell-limit}
        \eta_\ell(\mu)=\lim_{i \to \infty} \eta_\ell(\mu^i),
        \quad
        (\ell=0,1,2,\ldots).
    \end{equation}
\end{lemma}

\begin{proof}
    Let us suppose first that $\mu^i \weaktostar \mu$ 
    and $\abs{\mu^i} \weaktostar \lambda$ weakly* in $\Meas(\R^m; \R^N)$, resp., $\Meas(\R^m)$
    rather than just within $\Omega$.
    We denote by $E_f$ the discontinuity set of $f$,
    while $S_f$ stands for the approximate discontinuity set.
    Fubini's theorem and the fact that
    $S_{f}$ is an $\L^m$-negligible Borel set, imply that
    $\int \lambda(S_{\tau_x f_\ell}) \d x=0$.
    This and the non-negativity of $\lambda$ show that $\lambda(S_{\tau_x f_\ell})=0$ for \ae $x \in \R^m$.
    Since by assumption $E_f \subset S_f$,
    it follows that $\lambda(E_{\tau_x f_\ell})=0$,
    so that 
    (see, e.g., \cite[Proposition 1.62]{ambcosdal96})
    $\mu^i(\tau_x f_\ell) \to \mu(\tau_x f_\ell)$ for \ae $x \in \R^m$. 
    Likewise
    $\abs{\mu^i}(\tau_x f_\ell) \to \lambda(\tau_x f_\ell)$ for \ae $x \in \R^m$. 
    Since $\sup_i \abs{\mu^i}(\Omega) < \infty$,
    and $\Omega$ is bounded, an application of the dominated convergence
    theorem now yields
    \begin{equation}
        \label{eq:divergence-regul-average-conv}
        \lim_{i \to \infty}
        \int \abs{\mu^i(\tau_x f_\ell)} \d x 
        =
        \int \abs{\mu(\tau_x f_\ell)} \d x.
    \end{equation}
    By the lower semicontinuity of the total variation, recalling that
    \[
        \int \abs{\mu^i}(\tau_x f_\ell) \d x
        = \abs{\mu^i}(\R^m),
    \]
    this shows \eqref{eq:eta-ell-lsc} under the assumption that 
    the weak* convergences are in $\Meas(\R^m)$.

    Observe then that since $\{(f_\ell, \nu_\ell)\}_{\ell=0}^\infty$
    is a nested sequence of functions, $\{\eta_\ell(\mu)\}_{\ell=0}^\infty$ forms 
    a decreasing sequence (for any $\mu \in \Meas(\Omega)$).
    Indeed, as $f_\ell(x) = \int f_{\ell+1}(x - y) \d\nu_\ell(y)$
    and $\nu_\ell(\R^m)=1$ with $\nu_\ell \ge 0$, 
    using Lemma \ref{lemma:integral-norm-inside} we have
    \[
        \begin{split}
        \int \norm{\mu(\tau_x f_\ell)} \d x
        &
        = \int \adaptnorm{\int \mu(\tau_{x+y} f_{\ell+1}) \d\nu_\ell(y)} \d x
        \\
        &
        \le \int \int \norm{\mu(\tau_{x+y} f_{\ell+1})} \d\nu_\ell(y) \d x
        = \int \norm{\mu(\tau_{x} f_{\ell+1})} \d x.
        \end{split}
    \]
    Referring to \eqref{eq:def-eta-n},
    it follows that
    \begin{equation}
        \label{eq:eta-decreasing}
        \eta_\ell(\mu) \ge \eta_{\ell+1}(\mu).
    \end{equation}
    To show $\lambda=\abs{\mu}$,
    that is $\abs{\mu^i} \weaktostar \abs{\mu}$,
    we only have to show $\abs{\mu^i}(\Omega) \to \abs{\mu}(\Omega)$.
    To see the latter, we choose an arbitrary $\epsilon > 0$, and write
    \begin{equation}
        \label{eq:divergence-regul-diff}
        \abs{\mu}(\Omega)-\abs{\mu^i}(\Omega)
        =
        \eta_\ell(\mu)-\eta_\ell(\mu^i)
        + \int \abs{\mu(\tau_x f_\ell)}-\abs{\mu^i(\tau_x f_\ell)} \d x.
    \end{equation}
    Since $\eta_\ell \ge 0$, and
    using \eqref{eq:eta-decreasing} in
    \eqref{eq:eta-ki-bnd1} and \eqref{eq:eta-ki-bnd2},
    we now observe that taking $\ell$ large enough and $i_\ell$ 
    such that $K_{i_\ell}\ge\ell$, we have
    \[
        \sup\{\eta_\ell(\mu), \eta_\ell(\mu^{i_\ell}), \eta_\ell(\mu^{i_\ell+1}), \eta_\ell(\mu^{i_\ell+2}), \ldots\} \le \epsilon.
    \]
    Employing this in \eqref{eq:divergence-regul-diff},
    we deduce for any large enough $\ell$ and all $i \ge i_\ell$ that
    \[
        \adaptabs{\abs{\mu}(\Omega)-\abs{\mu^i}(\Omega)}
        \le 2\epsilon +
            \adaptabs{\int \abs{\mu(\tau_x f_\ell)}-\abs{\mu^i(\tau_x f_\ell)} \d x}.
    \]
    The integral term tends to zero as $i \to \infty$ by 
    \eqref{eq:divergence-regul-average-conv}.
    Therefore
    \[
        \lim_{i \to \infty} \adaptabs{\abs{\mu^i}(\Omega)-\abs{\mu}(\Omega)}
            \le 2\epsilon.
    \]
    Since $\epsilon > 0$ was arbitrary, 
    we conclude that $\lambda=\abs{\mu}$.
    Moreover, \eqref{eq:eta-ell-limit} follows
    from \eqref{eq:divergence-regul-average-conv} now.
    This concludes the proof of the lemma under the assumption 
    that the weak* convergences are in $\Meas(\R^m)$.
    
    If this assumption does not hold, we may still
    switch to a subsequence for which
    $\mu^{i_k} \weaktostar \bar \mu$
    weakly* in $\Meas(\R^m; \R^N)$
    for some $\bar \mu \in \Meas(\R^m; \R^N)$.
    But, since $\Omega$ is open, necessarily
    $\bar \mu \restrict \Omega = \mu$.
    Moreover, an application of the triangle inequality gives
    \[
        \eta_\ell(\mu)=\eta_\ell(\bar \mu \restrict \Omega)
        \le \eta_\ell(\bar \mu) \le \liminf_{k \to \infty} \eta_\ell(\mu^{i_k}).
    \]
    As this bound holds for every subsequence, we deduce 
    \eqref{eq:eta-ell-lsc}.
    Likewise, we have $\abs{\mu^{i_k}} \weaktostar \bar \lambda$
    weakly* in $\Meas(\R^m)$ for a common subsequence.
    Again $\bar \lambda \restrict \Omega = \lambda$.
    Since by the previous paragraphs $\abs{\bar\mu}=\bar\lambda$,
    this implies $\lambda=\abs{\mu}$.
\end{proof}

\begin{theorem}
    \label{thm:eta}
    Let $\Omega \subset \R^m$ be an open and bounded set,
    and $\{(f_\ell, \nu_\ell)\}_{\ell=0}^\infty$ a normalised nested
    sequence of functions. 
    Suppose the sequence $\{\mu^i\}_{i=0}^\infty$ in $\Meas(\Omega; \R^N)$
    converges weakly* to a measure $\mu \in \Meas(\Omega; \R^N)$,
    and satisfies $\sup_i \abs{\mu^i}(\Omega) + \eta(\mu^i) < \infty$.
    If also $\abs{\mu^i} \weaktostar \lambda$, then $\lambda=\abs{\mu}$.
    Moreover, each of the functionals 
    $\eta$ and $\eta_\ell$, ($\ell=0,1,2,\ldots$),
    is lower-semicontinuous with respect to the weak* 
    convergence of $\{\mu^i\}_{i=0}^\infty$.
\end{theorem}

\begin{proof}
    Only lower semicontinuity of $\eta$ demands a proof; the rest is immediate
    from Lemma \ref{lemma:eta-ki} with $K_i=i$, for instance.
    Also lower semicontinuity of $\eta$ follows simply from 
    Fatou's lemma and \eqref{eq:eta-ell-lsc}.
\end{proof}

\section{Proof of Theorem \ref{thm:compact}}
\label{app:thm:compact}

We prove here our result on the remedy by resorting to compact operators.

\begin{lemma} 
    \label{lemma:tv-compact-k}
    Let $\Omega$, $\Omega'$ be open domains,
    and suppose $K: \BVspace(\Omega) \to L^1(\Omega'; \R^m)$ is 
    a compact linear operator. Let $\phi: \nonnegR \to \nonnegR$ be
    lower semicontinuous. Then
    \[
        F(u) \defeq \int_{\Omega'} \phi(\norm{K u(x)}) \d x
    \]
    is lower semicontinuous with respect to weak* convergence
    in $\BVspace(\Omega)$.
\end{lemma}

\begin{proof}
    If $\{u^i\}_{i=1}^\infty \subset \BVspace(\Omega)$ converges
    weakly* to $u \in \BVspace(\Omega)$, then it is bounded in $\BVspace(\Omega)$. 
    Therefore, by the compactness of $K$, the 
    sequence $\{Ku^i\}_{i=1}^\infty$ has a subsequence,
    unrelabelled, which converges strongly to some $v \in L^1(\Omega'; \R^m)$.
    By the continuity of $K$, which follows from compactness,
    necessarily $v=Ku$.
    Now \cite[Theorem 5.9]{fonseca2007mmc} shows that
    \[
        F(u) \le \lim_{i \to \infty} F(u^i).
        \qedhere
    \]
\end{proof}

\begin{lemma}
    \label{lemma:meas-moll-conv}
    Let $\rho \in C_c^\infty(\R^n)$ and $\Omega \subset \R^n$ be bounded and open.
    Suppose $\mu^i \to \mu$ weakly* in $\Meas(\Omega; \R^m)$.
    Then $\rho * \mu^i \to \rho * \mu$ strongly in $L^\infty(\R^n)$.
\end{lemma}

\begin{proof}
    Let $K$ be a compact set such that $\Omega + \support \rho \subset K$.
    We have
    \[
        \norm{\rho * \mu^i}_{L^\infty(K; \R^m)} \le \norm{\rho}_{L^\infty(\R^n)} \abs{\mu^i}(\Omega)
    \]
    and
    \[
        \norm{\grad(\rho * \mu^i)}_{L^\infty(K; \R^n \times \R^m)} 
        =
        \norm{(\grad \rho) * \mu^i}_{L^\infty(K; \R^n \times \R^m)} 
        \le \norm{\grad \rho}_{L^\infty(\R^n; \R^n)} \abs{\mu^i}(\Omega).
    \]
    Thus $\{ \rho * \mu^i\}_{i=1}^\infty$ is uniformly bounded and equicontinuous.
    It follows from the Arzel{\'a}-Ascoli theorem that $\rho * \mu^i$ converges 
    uniformly (i.e., in $L^\infty(K; \R^m)$) to some $v \in C_c(K; \R^m)$.

    Let $\phi \in L^1(K; \R^m)$. Then by the weak* convergence we have
    \[
        \int_{\R^n} \iprod{\phi(x)}{(\rho * \mu^i)(x)} \d x 
        =
        \int_{\R^n} \iprod{(\phi * \rho)(x)}{\d \mu^i(x)}
        \to
        \int_{\R^n} \iprod{(\phi * \rho)(x)}{\d \mu(x)}
        =
        \int_{\R^n} \iprod{\phi(x)}{(\rho * \mu)(x)} \d x,
    \]
    as $i\to\infty$,
    so that $\rho * \mu^i \to \rho * \mu$ weakly in $L^\infty(K; \R^m)$.
    Then it holds that $\rho * \mu = v$, and the convergence is strong.
    Because $\support w \subset K$ for $w=\rho * \mu$ or $w=\rho * \mu^i$,
    the claim follows.
\end{proof}

\begin{proof}[Proof of Theorem \ref{thm:compact}]
    Suppose $\{u^i\}_{i=1}^\infty$ is a bounded sequence in $\BVspace(\Omega)$.
    We may extract a subsequence, 
    unrelabelled, such that $\{u^i\}_{i=1}^\infty$ is weakly*
    convergent in $\BVspace(\Omega)$ to some $u \in \BVspace(\Omega)$.
    By Lemma \ref{lemma:meas-moll-conv}, then
    \[
        D_\epsilon u^i \to D_\epsilon u \quad \text{strongly in } L^\infty(\Omega; \R^m).
    \]
    Weak* lower semicontinuity now follows from Lemma \ref{lemma:tv-compact-k}.
    
    Let then $u \in C^1(\Omega)$. The estimate 
    \begin{equation}
        \label{eq:tvphic-epsilon-bound}
        \limsup_{\epsilon\downto 0} \TVphic[\phi,\epsilon](u) \le \widetilde \TVphic(u),
        \quad (u \in C^1(\Omega)),
    \end{equation}
    follows from the proof of Lemma \ref{lemma:tvphic}.
    By subadditivity we also have
    \[
        \widetilde \TVphic(u) - \TVphic[\phi,\epsilon](u)
        \le
        \int_\Omega \phi(\norm{(\rho_\epsilon * \grad u)(x) - \grad u(x)}) \d x.
    \]
    Writing $g_\epsilon(x) \defeq \norm{(\rho_\epsilon * \grad u)(x) - \grad u(x)}$,
    we have $g_\epsilon(x) \to 0$ in $L^1(\R^n)$. 
    We may again proceed as in the proof of Lemma \ref{lemma:tvphic}
    to show
    \[
        \limsup_{\epsilon \downto 0}\bigl(\widetilde \TVphic(u) - \TVphic[\phi,\epsilon](u)\bigr) \le 0.
    \]
    This together with \eqref{eq:tvphic-epsilon-bound}
    proves \eqref{eq:tvphic-epsilon-conv}.
\end{proof}

\section{Proof of Theorem \ref{thm:sbveta}}
\label{app:thm:sbveta}

We now prove our result on the remedy based the SBV space.

\begin{proposition}
    \label{prop:remedy2-base}
    Let $\bar\eta$ and $p \in (1, \infty)$ be as in \eqref{eq:bareta}.
    Suppose $g^i \weakto g$ weakly in $L^p(\Omega; \R^m)$, 
    and that $\sup_i \bar \eta(g^i) < \infty$.
    Then $g^i \to g$ strongly in $L^p(\Omega; \R^m)$.
\end{proposition}

\begin{proof}
    Let $K$ be a compact set such that $\Omega + \support \rho_1 \subset K$,
    and set $M \defeq \sup_i \bar \eta(g^i)$. 
    We observe that $\bar \eta$ is lower semicontinuous with
    respect to weak convergence in $L^p$. Therefore $\eta(g) \le M$.
    As in the proof of Lemma \ref{lemma:eta-ki}, we observe that
    $\bar\eta_\ell \ge \bar\eta_{\ell+1}$, from which it again 
    follows that
    \begin{equation}
        \label{eq:eta-bar-univconv}
        \bar\eta_\ell(h) \to 0
        \quad
        \text{ as } \ell \to \infty
        \quad
        \text{ uniformly for } h \in \{g, g^1, g^2, g^3, \ldots\}.
    \end{equation}
    We then observe that as in Lemma \ref{lemma:meas-moll-conv}, we have
    \begin{equation}
        \label{eq:eta-bar-strongconv}
        \rho_{\epsilon_\ell} * g^i \to \rho_{\epsilon_\ell} * g
        \quad \text{strongly in } L^\infty(K; \R^m)
    \end{equation}
    for each $\ell \in \{1,2,3,\ldots\}$.
    Thus, it holds that
    \[
        \norm{g^i}_{L^p(\Omega; \R^m)}
        -
        \norm{g}_{L^p(\Omega; \R^m)}
        \le
        \eta_\ell(g^i)
        -
        \eta_\ell(g)
        +
        \norm{\rho_{\epsilon_\ell} * g^i-\rho_{\epsilon_\ell} * g}_{L^p(K; \R^m)}.
    \]
    Given $\delta > 0$, thanks to \eqref{eq:eta-bar-univconv}, we may find $\ell$ such that
    \[
        \norm{g^i}_{L^p(\Omega; \R^m)}
        -
        \norm{g}_{L^p(\Omega; \R^m)}
        \le
        2 \delta
        +
        \norm{\rho_{\epsilon_\ell} * g^i-\rho_{\epsilon_\ell} * g}_{L^p(K; \R^m)}.
    \]
    With $\ell$ fixed, we thus get by \eqref{eq:eta-bar-strongconv} that
    \[
        \limsup_{i \to \infty}
        \norm{g^i}_{L^p(\Omega; \R^m)}
        -
        \norm{g}_{L^p(\Omega; \R^m)}
        \le
        2 \delta.
    \]
    Since $\delta>0$ was arbitrary, and using weak lower semicontinuity of $\|\cdot\|_{L^p(\Omega; \R^m)}$, we deduce
    \[
        \lim_{i \to \infty} \norm{g^i}_{L^p(\Omega; \R^m)}
        =
        \norm{g}_{L^p(\Omega; \R^m)}.
    \]
    But for $p \in (1, \infty)$, strict convergence implies strong 
    convergence \cite{fonseca2007mmc}. This concludes the proof.
\end{proof}

\begin{proof}[Proof of Theorem \ref{thm:sbveta}]
    If $\{u^i\}_{i=1}^\infty$ is weakly* convergent in $\BVspace(\Omega)$,
    we may -- without loss of generality -- assume that $\sup_i F(u^i) < \infty$, for otherwise
    lower semicontinuity is obvious. 
    Then by the SBV Compactness Theorem \ref{thm:sbv-compactness},
    the convergences \eqref{eq:sbv-conv-1}--\eqref{eq:sbv-conv-4}
    hold for a subsequence. This also shows weak* convergence, 
    if it did not hold originally. Moreover, by the same theorem,
    $u \mapsto \int_{J_u} \psi(\theta_u(x)) \d \H^{m-1}(x)$ is
    lower semicontinuous with respect to this convergence.
    By \eqref{eq:sbveta-p-bound} we may further assume $\{\grad u^i\}_{i=1}^\infty$ 
    weakly convergent in $L^p(\Omega; \R^m)$. Proposition \ref{prop:remedy2-base}
    now shows strong convergence of $\{\grad u^i\}_{i=1}^\infty$ 
    in $L^p(\Omega; \R^m)$. The functional $F$ is lower semicontinuous
    with respect to all of these convergences, which yields lower 
    semicontinuity with respect to weak* convergence in $\BVspace(\Omega)$.
\end{proof}

\end{document}

%% file: texbase.tex
\usepackage[english]{babel}
\usepackage[labelfont=bf]{caption}
\usepackage[labelfont=bf]{subcaption}
\usepackage{comment}
\usepackage{footmisc}
\usepackage{bbm}
\usepackage{amsmath}
\usepackage{amssymb}
\usepackage{amsthm}
\usepackage{ifpdf}

\ifpdf
    \usepackage[pdftex]{hyperref}
\else
    \usepackage[hypertex]{hyperref}
\fi

\usepackage{color}
\definecolor{hrefcolor}{rgb}{0.0,0.5,0.8}
\definecolor{hlgreen}{rgb}{0,0.7,0}

\hypersetup{
   colorlinks=true,
   linkcolor=hrefcolor,
   citecolor=hlgreen,
   filecolor=hrefcolor,
   urlcolor=hrefcolor,
}

\ifdefined\tikz
\else
    \ifpdf
        \usepackage[pdftex]{graphicx}
    \else
        \usepackage{graphicx}
        
    \fi

\fi

\newenvironment{enumroman}
    {
     
     \begin{enumerate}
        \setlength{\leftmargin}{3.0em}
        \setlength{\labelwidth}{2.5em}
        \setlength{\labelsep}{0.5em}
    }
    {\end{enumerate}}

\newcounter{remcount}

\newtheorem{theorem}{Theorem}
\newtheorem{corollary}{Corollary}
\newtheorem{lemma}{Lemma}
\newtheorem{proposition}{Proposition}

\theoremstyle{definition}
\newtheorem{definition}{Definition}

\newtheorem*{assumption*}{Assumption}
\newtheorem{remark}{Remark}
\newtheorem*{remark*}{Remark}
\newtheorem*{definition*}{Definition}
\newtheorem{example}{Example}

\numberwithin{equation}{section}
\numberwithin{lemma}{section}
\numberwithin{theorem}{section}
\numberwithin{proposition}{section}
\numberwithin{definition}{section}
\numberwithin{remark}{section}
\numberwithin{example}{section}
\numberwithin{assumption}{section}
\numberwithin{algorithm}{section}
\numberwithin{corollary}{section}

\makeatletter
\AtBeginDocument{
    \hypersetup{
        pdftitle = {\@title},
        pdfauthor = {\@author}
    }
}
\makeatother

%% file: symbols-common.tex
\newcommand{\term}{\emph}

\newcommand{\field}[1]{\mathbb{#1}}
\newcommand{\N}{\mathbb{N}}
\newcommand{\Z}{\mathbb{Z}}

\newcommand{\R}{\field{R}}
\newcommand{\B}{B}

\newcommand{\norm}[1]{\|#1\|}
\newcommand{\adaptnorm}[1]{\left\|#1\right\|}

\newcommand{\abs}[1]{|#1|}
\newcommand{\adaptabs}[1]{\left|#1\right|}

\newcommand{\grad}{\nabla}

\newcommand{\Union}\bigcup
\newcommand{\Isect}\bigcap
\newcommand{\union}\cup
\newcommand{\isect}\cap
\newcommand{\bigunion}\bigcup
\newcommand{\bigisect}\bigcap

\newcommand{\defeq}{:=}

\newcommand{\downto}{\searrow}
\newcommand{\upto}{\nearrow}

\DeclareMathOperator{\closure}{cl}

\makeatletter
\def \uminus@sym{\setbox0=\hbox{$\cup$}\rlap{\hbox 
        to\wd0{\hss\raise0.5ex\hbox{$\scriptscriptstyle{-}$}\hss}}\box0}
    \def \uminus    {\mathrel{\uminus@sym}}
\makeatother

\newcommand{\mathvar}[1]{\textup{#1}}

%% file: symbols-ten.tex
\renewcommand{\tilde}{\widetilde}
\renewcommand{\ae}{a.e.~}
\newcommand{\iprod}[2]{\langle #1,#2\rangle}

\newcommand{\BVspace}{\mathvar{BV}}
\newcommand{\SBVspace}{\mathvar{SBV}}
\newcommand{\Eabs}{\mathcal{E}}

\newcommand{\Meas}{\mathcal{M}}

\renewcommand{\H}{\mathcal{H}}
\renewcommand{\L}{\mathcal{L}}
\newcommand{\restrict}{\llcorner}
\newcommand{\BD}{\partial}
\renewcommand{\d}{\,d} 

\newcommand{\TGV}{\mathvar{TGV}}
\newcommand{\TV}{\mathvar{TV}}

\DeclareMathOperator{\support}{supp}

\DeclareMathOperator{\divergence}{div}

\newcommand{\weakto}{\mathrel{\rightharpoonup}}
\makeatletter
\def \weaktostar@sym{\setbox0=\hbox{$\rightharpoonup$}\rlap{\hbox 
        to\wd0{\hss\raise1ex\hbox{$\scriptscriptstyle{*\,}$}\hss}}\box0}
    \def \weaktostar    {\mathrel{\weaktostar@sym}}
\makeatother

%% file: todo.tex
\usepackage{soul}

\newcommand{\TODO}[1]{{\sethlcolor{magenta}\color{white}\bfseries\em\hl{#1}}}

%% file: tvq-syms.tex
\def\TVq{\mathvar{TV}^q}
\def\TVqc{\mathvar{TV}^q_\mathrm{c}}
\def\TVqd{\mathvar{TV}^q_\mathrm{d}}
\def\phi{\varphi}
\newcommand{\TVphi}[1][\phi]{\mathvar{TV}^{#1}}
\newcommand{\TVphic}[1][\phi]{\mathvar{TV}^{#1}_\mathrm{c}}
\newcommand{\TVphics}[1][\phi]{\mathvar{TV}^{#1}_\mathrm{as}}
\newcommand{\TVphid}[1][\phi]{\mathvar{TV}^{#1}_\mathrm{d}}
\def\nonnegR{\R^{0,+}}
\def\pwc{\mathvar{pwc}}
\def\ICTV{\mathvar{ICTV}}

\newcommand{\Wc}{\mathcal{W}_\mathrm{c}}
\newcommand{\Wd}{\mathcal{W}_\mathrm{d}}
\newcommand{\Was}{\mathcal{W}_\mathrm{as}}

\renewcommand{\H}{\mathcal{H}}
\renewcommand{\d}{\,d} 

\DeclareMathOperator{\supp}{supp}

\renewcommand{\closure}[1]{\overline{#1}}
\def\ee{\text{\textsc{e}}}

%% file: img/dock-data.tex
\def\IMGthres{30}\def\IMGfullq{0.50}\def\IMGsmoothq{0.36}\def\IMGedgeq{1.44}\def\IMGlinlimalpha{0.059}\def\IMGlinq{0.32}\def\IMGlinoptq{0.42}\def\IMGlinoptlimalpha{0.050}\def\IMGlinoptthres{69}

%% file: img/parrot-data.tex
\def\IMGthres{20}\def\IMGfullq{0.46}\def\IMGsmoothq{0.68}\def\IMGedgeq{1.18}\def\IMGlinlimalpha{0.046}\def\IMGlinq{0.22}\def\IMGlinoptq{0.40}\def\IMGlinoptlimalpha{0.041}\def\IMGlinoptthres{73}

%% file: img/summer-data.tex
\def\IMGthres{30}\def\IMGfullq{0.87}\def\IMGsmoothq{0.24}\def\IMGedgeq{1.05}\def\IMGlinlimalpha{0.050}\def\IMGlinq{0.35}\def\IMGlinoptq{0.40}\def\IMGlinoptlimalpha{0.049}\def\IMGlinoptthres{40}

%% file: numerics.tex
\floatstyle{ruled}
\newfloat{Algorithm}{t}{loa}

Next we provide a numerical solver for the following TV$^\varphi$ model, possibly including the $\eta$-terms, in the discrete setting:
\begin{multline} \label{eq:dtvphi}
\min_u f(u):=\sum_{k\in\Omega_h}\Biggl(\alpha\varphi(|\nabla u(k)|) +\frac12|z(k)-u(k)|^2 
    \\ +\eta_0\sum_{l=1}^{N}\left(\sqrt{1+|\nabla u(k)|^2}-\sqrt{1+|\nabla_{\epsilon_l}u(k)|^2}\right)\Biggr).
\end{multline}
Here $\varphi$ is given by \eqref{eq:phi-linearised}, and $\alpha$, $\eta_0$ are manually chosen to balance the weights of the respective terms. 
For an image $u$ of resolution $n_1$-by-$n_2$, we set 
$h:=\sqrt{1/n_1n_2}$, $\Omega_d:=[0,1]^2 \cap (h\mathbb{Z}^2)$, and discretize the gradient by forward differences, i.e.~
\[
    \nabla u(k):=\left((u(k+e_1)-u(k))/h,~(u(k+e_2)-u(k))/h\right), \quad \text{for all} \quad k\in\Omega_d,
\]    
with homogeneous Dirichlet boundary condition.
Each $\nabla_{\epsilon_l}u:= \rho_{\epsilon_l}*\nabla u$ is defined through the convolution with a prescribed smoothing kernel $\rho_{\epsilon_l}$ ($\epsilon_l>0$). Here $\rho_{\epsilon_l}$ is specified as a two-dimensional Gaussian distribution of standard deviation $\epsilon_l$ centered at the origin.

To cope with the kink of the non-smooth $\varphi$ term at zero, we introduce a Huber-type local smoothing \cite{HiWu13_siims,HiWu14_coap} by replacing $\varphi$ in (\ref{eq:dtvphi}) with a continuously differentiable function $\varphi_\gamma$ with locally Lipschitz derivative. More specifically, let $0<\gamma\ll M$ be the smoothing parameter and $\varphi_\gamma:[0,\infty)\to[0,\infty)$ be 
defined by

\begin{equation} \nonumber
\varphi_\gamma(t)=
\begin{cases}
\varphi(t)-(\varphi(\gamma)-\frac12\varphi'(\gamma)\gamma) & \text{if }t\geq\gamma, \\
\frac{\varphi'(\gamma)}{\gamma^2}t^2 & \text{if }0\leq t<\gamma.
\end{cases}
\end{equation}
Thus, the resulting Huberized TV$^{\varphi_\gamma}$ model appears as
\begin{multline} \label{eq:sdtvphi}
\min_u f_\gamma(u):=\sum_{k\in\Omega_d}\Biggl(\alpha\varphi_\gamma(|\nabla u(k)|)+\frac12|z(k)-u(k)|^2+
    \\
    \eta_0\sum_{l=1}^{N}\left(\sqrt{1+|\nabla u(k)|^2}-\sqrt{1+|\nabla_{\epsilon_l}u(k)|^2}\right)\Biggr).
\end{multline}
For this problem, the first-order optimality condition reads:
\begin{equation} \label{eq:sele}
\left\{\begin{array}{l}
\alpha\nabla^\top p+u-z+\eta_0\sum_{l=1}^N\left(\nabla^\top\Big(\dfrac{\nabla u}{\sqrt{1+|\nabla u|^2}}\Big)-\nabla_{\epsilon_l}^\top\Big(\dfrac{\nabla_{\epsilon_l}u}{\sqrt{1+|\nabla_{\epsilon_l}u|^2}}\Big)\right)=0, \\
\psi(\max(|\nabla u|,\gamma))p=\nabla u,
\end{array}\right.
\end{equation}
where $p\in\left(\mathbb{R}^{|\Omega_d|}\right)^2$ is an auxiliary variable and $\psi:(0,\infty)\to(0,\infty)$ is defined by $\psi(t):=t/\varphi'(t)$. Note that $\psi$ is locally Lipschitz and monotonically increasing, and in the following we shall denote by $\partial\psi$ a subdifferential of $\psi$.
We remark that the consistency of the Huberized stationary points induced by (\ref{eq:sele}) towards the stationary points of the original model (\ref{eq:dtvphi}) 
was investigated 
in the previous work \cite{HiWu13_siims,HiWu14_coap}. 
Moreover, the system \eqref{eq:sele} is not differentiable in the classical sense.
Therefore, in the following we present a generalized Newton-type solver for computing a stationary point satisfying (\ref{eq:sele}). 

Given the current iterate $(u^i,p^i)$, our solver relies on the following regularized linear system arising from differentiating (\ref{eq:sele}) (and further straightforward manipulations; see \cite{HiWu13_siims,HiWu14_coap}):
\begin{equation} \notag
(H^i+\beta^iR^i)\delta u^i=-g^i,
\end{equation}
with
\begin{align}
m^i &:= \max(|\nabla u|,\gamma), \nonumber\\
\chi^i(k) &:=
\begin{cases}
1 & \text{if }|\nabla u(k)|\geq\gamma, \\
0 & \text{if }|\nabla u(k)|<\gamma, 
\end{cases} \nonumber\\
H^i &:= I+\alpha\nabla^\top\frac{1}{\psi(m^i)}\left(I-\frac{\chi^i\partial\psi(m^i)}{2m^i}\left((p^i)(\nabla u^i)^\top+(\nabla u^i)(p^i)^\top\right)\right)\nabla \nonumber\\
&\quad\quad +\eta_0\sum_{l=1}^N\left(\nabla^\top\frac{1}{\sqrt{1+|\nabla u^i|^2}}\left(I-\frac{(\nabla u^i)(\nabla u^i)^\top}{1+|\nabla u^i|^2}\right)\nabla \right.\nonumber\\
& \quad\quad \left.-\nabla_{\epsilon_l}^\top\frac{1}{\sqrt{1+|\nabla u^i|^2}}\left(I-\frac{(\nabla u^i)(\nabla u^i)^\top}{1+|\nabla u^i|^2}\right)\nabla_{\epsilon_l} \right), \nonumber\\
R^i &:= \varepsilon I+\alpha\nabla^\top\frac{\chi^i\partial\psi(m^i)}{2\psi(m^i)m^i}\left((p^i)(\nabla u^i)^\top+(\nabla u^i)(p^i)^\top\right)\nabla \nonumber\\
&\quad\quad +\eta_0\sum_{l=1}^N\nabla_{\epsilon_l}^\top\frac{1}{\sqrt{1+|\nabla u^i|^2}}\left(I-\frac{(\nabla u^i)(\nabla u^i)^\top}{1+|\nabla u^i|^2}\right)\nabla_{\epsilon_l}, \nonumber\\
g^i &:=\nabla f_\gamma(u^i)=\alpha\nabla^\top\Big(\frac{\nabla u^i}{\psi(m^i)}\Big) +u^i-z
    \\
    & \hspace{15ex} +\eta_0\sum_{l=1}^N\left(\nabla^\top\Big(\dfrac{\nabla u^i}{\sqrt{1+|\nabla u^i|^2}}\Big)-\nabla_{\epsilon_l}^\top\Big(\dfrac{\nabla_{\epsilon_l}u^i}{\sqrt{1+|\nabla_{\epsilon_l}u^i|^2}}\Big)\right). \nonumber
\end{align}
Here $H^i$ represents a (modified) generalized Hessian matrix of $f_\gamma$ at $u^i$, while $R^i$, with an
arbitrarily fixed $0<\varepsilon\ll\alpha$, serves as a structural Hessian
regularization. 
Note that $H^i$ may not be positive definite at the iterate $u^i$. 
For this reason, the regularization weight $\beta^i$ is automatically tuned by a
trust-region based mechanism; see steps 8--20.
Further, whenever $\beta^i=1$, the regularized Hessian, 
i.e.~$H^i+\beta^iR^i$, 
is positive definite. Consequently, our $\beta^i$-update scheme guarantees $\delta u^i$ to be a descent direction for $f_\gamma$ at $u^i$,
and thus the overall iterative scheme can be globalised by, e.g., the Wolfe-Powell line
search \cite{DeSc96}; see step 21 in Algorithm \ref{algrrn}. 
Moreover, following the algorithm development in \cite{HiWu13_siims,HiWu14_coap}, one can show that $\beta^i$ asymptotically vanishes as $u^i$ approaches a stationary point where a certain type of second-order sufficient optimality condition is satisfied. Thus, local superlinear convergence
can be attained. The overall algorithm is detailed in Algorithm \ref{algrrn} below. The following parameters associated with Algorithm \ref{algrrn} are specified throughout our experiments: $c=1$, $\underline{\rho}=0.25$, $\bar\rho=0.75$, $\underline{\kappa}=0.5$, $\bar{\kappa}=2$, $\epsilon_d=10^{-10}$, $\tau_1=0.1$, $\tau_2=0.9$, $\gamma=0.001$, $\epsilon_{r}=0.01$.
Algorithm \ref{algrrn} is terminated once $\|\nabla f_\gamma(u^i)\|/\|\nabla f_\gamma(u^0)\|$ drops below $10^{-7}$.

\begin{Algorithm}[!ht]
\caption{Superlinearly convergent Newton-type method for $\TV^{\phi_\gamma}$ denoising}
\label{algrrn} {\ } 
\begin{algorithmic}[1]
\REQUIRE 
$c>0$, $0<\underline{\rho}\leq\bar{\rho}<1$, $0<\underline{\kappa}<1<\bar{\kappa}$, $\epsilon_d>0$, $0<\tau_1<1/2$, $\tau_1<\tau_2<1$, $0<\gamma\ll1$, $0<\epsilon_{r}<1$. 
\STATE Initialize the iterate $(u^0,p^0)$, the regularization weight $\beta^{0}\geq 0$, and the trust-region radius $r^{0}>0$. Set $i:=0$.
\REPEAT
\STATE Generate $H^i$, $R^i$, and $g^i$ at the current iterate $(u^i,p^i)$. \label{algm14}
\STATE Solve the linear system $(H^i+\beta^i R^i)\delta u^i=-g^i$ (inexactly) for $\delta u^i$
by the conjugate gradient (CG) method up to the residual tolerance $\epsilon_{r}$, or detect the non-positive definiteness of $H^i+\beta^i R^i$ during the CG iterations. \label{algm1}
\IF{$H^i+\beta^i R^i$ is not positive definite \OR $-((g^i)^\top \delta u^i)/(\|g^i\|\|\delta u^i\|)<\epsilon_d$} \label{algm6}
\STATE Set $\beta^i:=1$, and return to step \ref{algm1}. \label{algm12}
\ENDIF \label{algm7}
\IF{$\beta^i=1$ \AND $(\delta u^i)^\top R^i\delta u^i>(r^i)^2$} \label{algm8}
\STATE Set $r^i:=\sqrt{(\delta u^i)^\top R^i\delta u^i}$, $\beta^{i+1}:=1$, and go to step \ref{algm2}.
\ELSE 
\STATE Set $\beta^{i+1}:=\max(\min(\beta^i+c^{-1}((\delta u^i)^\top R^i\delta u^i)-(r^i)^2),1),0)$. \label{algm10}
\ENDIF \label{algm9}
\STATE Evaluate $\rho^i:= \left(f_{\gamma}(u^i+\delta u^i)-f_{\gamma}(u^i)\right)/\left((g^i)^\top \delta u^i+(\delta u^i)^\top H^i\delta u^i/2\right).$ \label{algm2}
\IF{$\rho^i<\underline{\rho}$}
\STATE Set $r^{i+1}:=\underline{\kappa}r^i$.
\ELSIF{$\rho^i>\bar{\rho}$}
\STATE Set $r^{i+1}:= \bar{\kappa}r^i$.
\ELSE
\STATE Set $r^{i+1}:=r^i$.
\ENDIF
\STATE Determine the step size $a^i$ along the search direction $\delta u^i$ such that $u^{i+1}=u^i+a^i\delta u^i$ satisfies the following Wolfe-Powell conditions: \label{algm5}
\vspace{-.1cm}
\[
\begin{cases}
f_{\gamma}(u^{i+1}) \leq f_{\gamma}(u^i)+\tau_1a^i(g^i)^\top \delta u^i, \\
\nabla f_{\gamma}(u^{i+1})^\top \delta u^i \geq \tau_2 (g^i)^\top \delta u^i. 
\end{cases}
\]
\vspace{-.6cm}
\STATE 
Generate the next iterate: 
\begin{align}
u^{i+1} &:=u^i+a^i \delta u^i, \nonumber\\[-2pt]
p^{i+1} &:= \frac{1}{\psi(m^i)}\left(\nabla u^i+\nabla\delta u^i-\frac{\chi^i\partial\psi(m^i)}{2m^i}\left((p^i)(\nabla u^i)^\top+(\nabla u^i)(p^i)^\top\right)\nabla\delta u^i\right). \nonumber
\end{align}
\vspace{-.3cm}
\STATE Set $i:= i+1$. 
\UNTIL{the stopping criterion is fulfilled.}
\end{algorithmic}
\end{Algorithm}